\documentclass{amsart}

\pdfoutput=1

\usepackage{amsmath}
\usepackage{amssymb}
\usepackage{bbm}
\usepackage{mathrsfs}
\usepackage{hyperref}

\usepackage[T1]{fontenc}
\usepackage[utf8]{inputenc}

\usepackage{graphicx}

\numberwithin{equation}{section}
\numberwithin{figure}{section}

\newcommand{\R}{\mathbb{R}}

\newcommand{\e}{\operatorname{e}}
\newcommand{\dd}{\,{\mathrm d}}
\newcommand{\db}{{\mathrm d}}

\newcommand{\im}{\operatorname{i}}

\newcommand{\car}{\Gamma}
\newcommand{\drift}{b}
\newcommand{\SU}{{\rm SU}(2)}
\newcommand{\SL}{{\rm SL}(2,\R)}
\newcommand{\Hs}{\mathbb{H}}
\newcommand{\hess}{\operatorname{Hess}}

\newcommand{\eps}{\varepsilon}
\newcommand{\pt}{\partial}

\newtheorem{lemma}{Lemma}[section]

\newtheorem{propn}[lemma]{Proposition}
\newtheorem{thm}[lemma]{Theorem}

\newtheorem*{defn*}{Definition}
\newtheorem{remark0}[lemma]{Remark}
\newtheorem{eg0}[lemma]{Example}

\newenvironment{remark}{\begin{remark0}\rm}{\hspace*{\fill} $\square$
                        \end{remark0}}

\author[D. Barilari, U. Boscain, D. Cannarsa, K. Habermann]
{Davide Barilari, Ugo Boscain, Daniele Cannarsa, Karen Habermann}

\address{Davide Barilari, Universit{\'e} de Paris, Sorbonne
  Universit{\'e}, CNRS, Institut de Math{\'e}matiques de Jussieu-Paris
  Rive Gauche, F-75013 Paris, France.}
\email{davide.barilari@imj-prg.fr}

\address{Ugo Boscain, CNRS, Laboratoire Jacques-Louis Lions, Sorbonne
  Universit{\'e}, Universit{\'e} de Paris, Inria, F-75005 Paris, France.}
\email{ugo.boscain@upmc.fr}

\address{Daniele Cannarsa, Universit{\'e} de Paris, Sorbonne
  Universit{\'e}, CNRS, Inria, Institut de Math{\'e}matiques de
  Jussieu-Paris Rive Gauche, F-75013 Paris, France.}
\email{daniele.cannarsa@imj-prg.fr}

\address{Karen Habermann, Laboratoire Jacques-Louis Lions, Sorbonne
  Universit{\'e}, Universit{\'e} de Paris, CNRS, Inria, F-75005 Paris,
  France.}
\email{karen.habermann@upmc.fr}

\title[Stochastic processes on surfaces in 3D contact sub-Riemannian
manifolds]{Stochastic processes on surfaces in three-dimensional
  contact sub-Riemannian manifolds}
\begin{document}
\begin{abstract}
  We are concerned with stochastic processes on surfaces in
  three-dimensional contact sub-Riemannian manifolds. Employing the
  Riemannian approximations to the sub-Riemannian manifold which make
  use of the Reeb vector field, we obtain a second order partial
  differential operator on the surface arising as the limit of
  Laplace--Beltrami operators. The stochastic process
  associated with the limiting operator moves along the characteristic
  foliation induced on the surface by the contact distribution. 
  We show that for this stochastic process elliptic characteristic points
  are inaccessible, while hyperbolic characteristic points are
  accessible from the separatrices. We illustrate the results with
  examples and we identify canonical surfaces in the Heisenberg group,
  and in ${\rm SU}(2)$ and ${\rm SL}(2,\R)$ equipped
  with the standard sub-Riemannian contact structures
  as model cases for this setting.
  Our techniques further allow us to derive an expression for
  an intrinsic Gaussian curvature of a surface in a
  general three-dimensional contact sub-Riemannian manifold.
\end{abstract}

\maketitle
\thispagestyle{empty}
\section{Introduction}
The study of surfaces in three-dimensional contact manifolds has found
a lot of interest, amongst others, since the so-called oriented
singular foliation on the surface provides an important invariant used
to classify contact structures, see Abbas and
Hofer~\cite[Chapter~3]{abbas}, Geiges~\cite[Chapter~4]{geiges}, and
Giroux~\cite{giroux1,giroux2}. In recent years, there has been an
increased activity in studying surfaces in three-dimensional contact manifolds
whose contact distributions additionally carry a metric.
Balogh~\cite{balogh03}
analyses the Hausdorff dimension of the so-called characteristic
set of a hypersurface in the Heisenberg group. Balogh, Tyson and
Vecchi~\cite{balogh} define an intrinsic Gaussian curvature for
surfaces in the Heisenberg group and an intrinsic signed geodesic
curvature for curves on surfaces to obtain a Gauss--Bonnet theorem in
the Heisenberg group. Veloso~\cite{veloso} extends the results
in~\cite{balogh} to general three-dimensional contact manifolds for
non-characteristic surfaces. Danielli, Garofalo and
Nhieu~\cite{danielli} discuss the local summability of the
sub-Riemannian mean curvature of surfaces in the Heisenberg group.
The contribution of this paper is to introduce a canonical stochastic
process on a given surface in a three-dimensional contact
manifold whose contact distribution is equipped with a metric, to
analyse properties of the induced stochastic process and to identify
model cases for this setting.

Let $(M,D,g)$ be a three-dimensional contact sub-Riemannian manifold,
that is, we consider a three-dimensional manifold $M$ which is equipped with a
sub-Riemannian structure $(D,g)$ that is contact. A sub-Riemannian
structure on a manifold $M$ consists of a bracket generating
distribution $D\subset TM$ and a smooth fibre inner product $g$ defined on
$D$. Such a sub-Riemannian structure is said to be contact if the
distribution $D$ is a contact structure on $M$. Under the
assumption that $D$ is coorientable, the latter means that there exists
a global one-form $\omega$ on $M$ satisfying $\omega\wedge\db\omega\not=0$
and such that $D=\operatorname{ker}\omega$. The one-form $\omega$ is
called a contact form and the pair $(M,D)$ is called a contact
manifold. Throughout, we choose the contact form $\omega$ to be
normalised such that $\db\omega|_D=-\operatorname{vol}_g$ for
$\operatorname{vol}_g$ denoting the Euclidean volume form on $D$
induced by the fibre inner product $g$. Associated with the contact
form $\omega$, we have the Reeb vector field $X_0$ which is the unique
vector field on $M$ satisfying $\db\omega(X_0,\cdot)\equiv 0$ and
$\omega(X_0)\equiv 1$.

Let $S$ be an orientable surface embedded in the contact manifold $(M,D)$. We
call a point $x\in S$ a characteristic point of $S$ if the contact
plane $D_x$ coincides with the tangent space $T_xS$. Note that
characteristic points are also called singular points,
cf.~\cite{abbas} and \cite{geiges}. We denote the set
of all characteristic points of $S$ by $\car(S)$. If $x\in S$ is not a
characteristic point then $D_x$ and $T_xS$ intersect in a
one-dimensional subspace. These subspaces induce a singular
one-dimensional foliation on $S$, that is, an equivalence class of
vector fields which differ by a strictly positive or strictly negative
function. This foliation is
called the characteristic foliation of $S$ induced by the contact
structure $D$. We see that the canonical stochastic process we define
on the surface $S$ moves along the characteristic foliation.
This process does not hit elliptic characteristic
points, whereas a hyperbolic characteristic point is hit subject to an
appropriate choice of the starting point.
In the dynamical systems terminology, an elliptic point
corresponds to a node or a focus, and a hyperbolic point is called
a saddle, see Robinson~\cite{robinson}.

To construct the canonical stochastic process on $S$, we
consider the Riemannian approximations to the sub-Riemannian
manifold $(M,D,g)$ which make use of the Reeb vector field $X_0$.
For $\eps>0$, the Riemannian approximation to $(M,D,g)$ defined uniquely
by requiring $\sqrt{\eps}X_0$ to be unit-length and to be orthogonal
to the distribution
$D$ everywhere induces a Riemannian metric $g_\eps$ on $S$.
This gives rise to the two-dimensional Riemannian manifold $(S,g_\eps)$
and its Laplace--Beltrami operator~$\Delta_\eps$. We show that
the operators $\Delta_\eps$ converge
uniformly on compacts in $S\setminus\car(S)$ to an
operator $\Delta_0$ on $S\setminus\car(S)$, and we study
the stochastic process on $S\setminus\car(S)$ whose generator is
$\frac{1}{2}\Delta_0$.

To simplify the presentation of the paper, we shall assume that the
distribution $D$ is trivialisable, that is, globally generated by a
pair of vector fields, and we choose vector fields $X_1$ and $X_2$
such that $(X_1,X_2)$ is an oriented orthonormal frame for $D$ with
respect to the fibre inner product $g$. By the Cartan formula and due
to $\db\omega|_D=-\operatorname{vol}_g$, we have
\begin{equation*}
  \omega([X_1,X_2])=-\db\omega(X_1,X_2)=1\;.
\end{equation*}
Since $X_0$ is the Reeb vector field, we obtain
\begin{equation*}
  \omega([X_0,X_i])=-\db\omega(X_0,X_i)=0
  \quad\mbox{for }i\in\{1,2\}\;.  
\end{equation*}
It follows that there exist functions
$c_{ij}^1,c_{ij}^2\colon M\to\R$, for $i,j\in\{0,1,2\}$, such that
\begin{align}
  [X_1,X_2]&=c_{12}^1X_1+c_{12}^2X_2+X_0\;,\\
  [X_0,X_1]&=c_{01}^1X_1+c_{01}^2X_2\;,\label{defn:c01}\\
  [X_0,X_2]&=c_{02}^1X_1+c_{02}^2X_2\;\label{defn:c02}.
\end{align}
In particular, the vector fields $X_1$, $X_2$ and $[X_1,X_2]$ on $M$ are
linearly independent everywhere.
The Riemannian approximation to $(M,D,g)$ for $\eps>0$ is
then obtained by requiring $(X_1,X_2,\sqrt{\eps} X_0)$ to be a
global orthonormal frame.
We further suppose that the surface $S$ embedded in $M$ is given by
\begin{equation}\label{defn:Swu}
  S=\{x\in M: u(x)=0\}
  \quad\mbox{for }u\in C^2(M)
  \mbox{ with }\db u\not=0
  \mbox{ on } S\;.
\end{equation}
While this might define a surface consisting of multiple connected
components, we could always restrict our attention to a single
connected component. A point $x\in S$ is a characteristic point
if and only if $(X_1u)(x)=(X_2u)(x)=0$, that is,
\begin{equation}\label{cri:char}
  x\in\car(S)
  \quad\mbox{if and only if}\quad
  \left((X_1u)(x)\right)^2+\left((X_2u)(x)\right)^2=0\;.
\end{equation}
Consequently, the characteristic set $\car(S)$ is a closed subset of $S$.
With $\hess u$ denoting the horizontal Hessian of $u$
defined by
\begin{equation}\label{defn:hess}
  \hess u=
  \begin{pmatrix}
    X_1X_1 u & X_1X_2 u\\
    X_2X_1 u & X_2X_2 u
  \end{pmatrix}\;,
\end{equation}
we can classify the characteristic points of $S$ as follows. 
\begin{defn*}
A characteristic 
point $x\in\car(S)$ is called non-degenerate if
$\det((\hess u)(x))\not=0$,
it is called elliptic if $\det((\hess u)(x))>0$, and it is
called hyperbolic if $\det((\hess u)(x))<0$.
\end{defn*}
With the notations introduced above,
we can explicitly write down the expression of a
unit-length representative of the characteristic foliation
of $S$ induced by the contact structure $D$. Let $\widehat{X}_S$ be the
vector field on $S\setminus\car(S)$ defined by
\begin{equation}\label{defn:E1}
  \widehat{X}_S = \frac{(X_2u)X_1-(X_1u)X_2}{\sqrt{(X_1u)^2+(X_2u)^2}}\;.
\end{equation}
Note that while $\widehat{X}_S$ is expressed in terms of $X_1,X_2$ and
$u$, it only depends on the sub-Riemannian manifold $(M,D,g)$, the
embedded surface $S$ and a choice of sign. It is a vector field on
$S\setminus\car(S)$ whose vectors have unit length and lie in
$D|_S\cap TS$ with a continuous choice of sign. In particular, the
vector field $\widehat{X}_S$ remains unchanged if $u$ is multiplied by
a positive function. Let $\drift\colon S\setminus\car(S)\to \R$ be the
function given by
\begin{equation}\label{defn:drift}
  \drift=\frac{X_0u}{\sqrt{(X_1u)^2+(X_2u)^2}}\;.
\end{equation}
Similarly to the vector field $\widehat{X}_S$, the function $b$ can
be understood intrinsically. Let $\widehat{X}^{\perp}_{S}$
be such that
$(\widehat{X}_S,\widehat{X}^{\perp}_{S})$ is an oriented orthonormal
frame for $D|_{S\setminus\car(S)}$. The function $b$ is then uniquely
given by requiring $b\widehat{X}^{\perp}_{S}-X_{0}$ to be a vector
field on $S\setminus\car(S)$. Set
\begin{equation}\label{limitLB}
  \Delta_0=\widehat{X}_S^2+\drift \widehat{X}_S\;,
\end{equation}
which is a second order partial differential operator on
$S\setminus\car(S)$. 
The operator $\Delta_0$ is invariant under
multiplications of $u$ by functions which do not change its zero
set. As stated in the theorem below, it arises
as the limiting operator of the Laplace--Beltrami operators
$\Delta_\eps$ in the limit $\eps\to 0$.
\begin{thm}\label{thm:conv_of_LB}
  For any
  twice differentiable function $f\in C_c^2(S\setminus \car(S))$
  compactly supported in $S\setminus \car(S)$, the functions
  $\Delta_\eps f$ converge uniformly on $S\setminus \car(S)$ to
  $\Delta_0 f$ as $\eps\to 0$.
\end{thm}
Since the theorem above only concerns twice
differentiable functions of compact support in $S\setminus \car(S)$,
we do not have to put any additional assumptions on the set of
characteristic points of $S$. 

Following the definition in Balogh, Tyson and Vecchi~\cite{balogh} for
surfaces in the Heisenberg group, we introduce an intrinsic Gaussian
curvature $K_0$ of a surface in a general three-dimensional
contact sub-Riemannian manifold as the limit as $\eps\to 0$ of the
Gaussian
curvatures $K_\eps$ of the Riemannian manifolds $(S,g_{\eps})$. To
derive the expression given in the following proposition, we employ
the same orthogonal frame exhibited to prove
Theorem~\ref{thm:conv_of_LB}.
\begin{propn}\label{prop:gausscurv} Uniformly
  on compact subsets of $S\setminus\car(S)$, we have
  \begin{equation*}
    K_0:=\lim_{\eps\to 0} K_\eps
    =-\widehat{X}_S(b)-b^2\;.
  \end{equation*}
\end{propn}
We now consider the canonical stochastic process on $S\setminus\car(S)$
whose generator is $\frac{1}{2}\Delta_0$. Assuming that it
starts at a fixed point then, up to explosion, the
process moves along the unique leaf of the characteristic
foliation picked out by the starting point. As shown by the
next theorem and the following proposition,
for this stochastic process,
elliptic characteristic points are inaccessible, while hyperbolic
characteristic points are accessible from the separatrices.
Recall that what we call hyperbolic characteristic points
are known as saddles in the dynamical systems literature, whereas what
is referred to as hyperbolic points in their language are
non-degenerate characteristic points in our terminology.
\begin{thm}\label{thm:alive}
  The set of elliptic characteristic points in a surface $S$ embedded
  in $M$ is inaccessible for the stochastic process
  with generator $\frac{1}{2}\Delta_0$ on $S\setminus\car(S)$.
\end{thm}
In Section~\ref{sec:hyperpara}, we discuss an example of a surface in the
Heisenberg group whose induced stochastic process is killed in finite
time if started along the separatrices of the characteristic
point. Indeed, this
phenomena always occurs in the presence of a hyperbolic characteristic
point.
\begin{propn}\label{propn:hyperbolic}
  Suppose that the surface $S$ embedded in $M$ has a hyperbolic
  characteristic point. Then the stochastic process having generator
  $\frac{1}{2}\Delta_0$ and started on the separatrices of
  the hyperbolic characteristic point reaches that characteristic point
  with positive probability.
\end{propn}
The Sections~\ref{sec:heisenberg} and \ref{sec:model} are devoted to
illustrating the various behaviours the canonical stochastic process
induced on the surface $S$ can show. Besides illustrating
Proposition~\ref{propn:hyperbolic}, we show that three classes of
familiar stochastic processes arise when considering a natural choice for
the surface $S$ in the three classes of model spaces for three-dimensional
sub-Riemannian structures, which are the
Heisenberg group $\Hs$, and the special unitary group $\SU$ and the
special linear group $\SL$ equipped with sub-Riemannian
contact structures for fibre inner products differing by a constant
multiple. In all these cases, the orthonormal frame
$(X_1,X_2)$ for the distribution $D$ is formed by two left-invariant
vector fields which together with the Reeb vector field $X_0$ satisfy,
for some $\kappa\in\R$, the commutation relations
\begin{equation*} 
  [X_{1},X_{2}]=X_{0}\;,\quad
  [X_{0},X_{1}]=\kappa X_{2}\;, \quad
  [X_{0},X_{2}]=-\kappa X_{1}\;,
\end{equation*}
with $\kappa=0$ in the Heisenberg group, $\kappa>0$ in $\SU$ and
$\kappa<0$ in $\SL$.
Associated with each of these Lie groups and their Lie algebras, we
have the group exponential map $\exp$ for which we identify a
left-invariant vector field with its value at the origin.
\begin{thm}\label{thm:model}
  Fix $\kappa\in\R$. For $\kappa\not=0$,
  let $k\in\R$ with $k>0$ be such that
  $|\kappa|=4k^2$. Set $I=(0,\frac{\pi}{k})$ if $\kappa>0$ and
  $I=(0,\infty)$ otherwise.
  In the model space for three-dimensional
  sub-Riemannian structures corresponding to $\kappa$,
  we consider the embedded surface $S$ parameterised as
  \begin{equation*}
    S=\{\exp(r\cos \theta X_{1}+r\sin \theta X_{2}) :
    r\in I\mbox{ and } \theta\in [0,2\pi)\}\;.
  \end{equation*}
  Then the limiting operator $\Delta_0$ on $S$ is given by
  \begin{equation*}
    \Delta_0=\frac{\pt^2}{\pt r^2}+
    \drift\left(r\right)\frac{\pt}{\pt r}\;,
  \end{equation*}
  where
  \begin{equation*}
    \drift(r)=
    \begin{cases}
      2k\cot(k r)  &\mbox{if }\kappa=4k^2\\
      \frac{2}{r}  &\mbox{if }\kappa=0\\
      2k\coth(k r) &\mbox{if }\kappa=-4k^2\\
    \end{cases}\;.
  \end{equation*}
  The stochastic process induced by the operator
  $\frac{1}{2}\Delta_0$ moving along the leaves of the
  characteristic foliation of $S$
  is a Bessel process of order $3$ if $\kappa=0$, a Legendre
  process of order $3$ if $\kappa>0$ and a hyperbolic Bessel process of
  order $3$ if $\kappa<0$.
\end{thm}
Notably, the stochastic
processes recovered in the above theorem
are all related to one-dimensional Brownian motion by
the same type of
Girsanov transformation, with only the sign of a parameter distinguishing
between them. For the details, see Revuz and Yor \cite[p.~357]{revuz}.
A Bessel process of order $3$ arises by conditioning a one-dimensional
Brownian motion started on the positive real line to never hit the
origin, whereas a Legendre process of order $3$ is obtained by
conditioning a Brownian motion started inside an interval to never 
hit either endpoint of the interval. The examples making up
Theorem~\ref{thm:model} can be considered as model cases for our
setting, and all of them illustrate Theorem~\ref{thm:alive}.

Notice that the limiting operator we obtain on the leaves is not the
Laplacian associated with the metric structure restricted to the
leaves as the latter has no drift term. However, the operator
$\Delta_0$ restricted to a leaf can be considered as a weighted
Laplacian. For a smooth measure $\mu=h^2 \dd x$ on an interval $I$ of the
Euclidean line $\R$, the weighted Laplacian applied to a scalar
function $f$ yields
\begin{equation*}
  \operatorname{div}_\mu\left(\frac{\pt f}{\pt x}\right)=
  \frac{\pt^2f}{\pt x^2}+
  \frac{2 h'(x)}{h(x)}\frac{\pt f}{\pt x}\;.
\end{equation*}
In the model cases above, we have
\begin{equation*}
  h(r)=
  \begin{cases}
      \sin\left(k r\right)  &\mbox{if }\kappa=4k^2\\
      r  &\mbox{if }\kappa=0\\
      \sinh\left(k r\right) &\mbox{if }\kappa=-4k^2\\    
  \end{cases}\;.
\end{equation*}

We prove Theorem~\ref{thm:conv_of_LB}
and Proposition~\ref{prop:gausscurv}
in Section~\ref{sec:general}, where the proof of the theorem
relies on the expression of $\Delta_\eps|_{S\setminus\car(S)}$
given in Lemma~\ref{lem:LB_eps} in terms of an orthogonal frame for
$T(S\setminus\car(S))$. In Section~\ref{sec:stochastic}, we prove
Theorem~\ref{thm:alive} and Proposition~\ref{propn:hyperbolic}
using Lemma~\ref{lem:b4focus} and Lemma~\ref{lem:b4rest}, which
expand the function
$\drift\colon S\setminus\car(S)\to \R$ from (\ref{defn:drift})
in terms of the arc length along the integral curves of
$\widehat{X}_S$.
The
results are illustrated in the last two sections. In
Section~\ref{sec:heisenberg}, we study quadric surfaces in the
Heisenberg group, whereas in Section~\ref{sec:model}, we consider 
canonical surfaces in $\SU$ and $\SL$
equipped with the standard sub-Riemannian contact structures.
The examples establishing Theorem~\ref{thm:model}
are discussed in Section~\ref{sec:para}, Section~\ref{sec:SU} and
Section~\ref{sec:SL}, with a unified viewpoint presented in
Section~\ref{s:unif}.

\proof[Acknowledgement]
This work was supported by the Grant ANR-15-CE40-0018 ``Sub-Riemannian
Geometry and Interactions'' of the French ANR. The third author is
supported by grants from R{\'e}gion Ile-de-France. The fourth author is
supported by the Fondation Sciences Math{\'e}matiques de Paris. All
four authors would like to thank Robert Neel for illuminating discussions.

\section{Family of Laplace--Beltrami operators on
  the embedded surface}
\label{sec:general}
We express the Laplace--Beltrami operators $\Delta_\eps$
of the Riemannian manifolds $(S,g_\eps)$
in terms of
two vector fields on the surface $S$ which are orthogonal for
each of the Riemannian approximations employing the Reeb vector field.
Using these expressions of the Laplace--Beltrami operators
$\Delta_\eps$ where only the coefficients
and not the vector fields depend on $\eps>0$, we prove
Theorem~\ref{thm:conv_of_LB}. The orthogonal frame exhibited further
allows us to establish Proposition~\ref{prop:gausscurv}.

For a vector field $X$ on the manifold $M$, the property
$Xu|_S\equiv 0$ ensures that $X(x)\in T_xS$ for all $x\in S$.
Therefore, we see that $F_1$ and $F_2$ given by
\begin{align}
  F_1&=\frac{(X_2u)X_1-(X_1u)X_2}{\sqrt{(X_1u)^2+(X_2u)^2}}\label{defn:F1}
  \qquad\mbox{and}\\
  F_2&=\label{defn:F2}
       \frac{(X_0u)(X_1u)X_1+(X_0u)(X_2u)X_2}{(X_1u)^2+(X_2u)^2}-X_0\;,
\end{align}
are indeed well-defined vector fields on $S\setminus \car(S)$
due to~(\ref{cri:char}) and because we have
$F_1u|_{S\setminus \car(S)}\equiv 0$ as well as 
$F_2u|_{S\setminus \car(S)}\equiv 0$. Here, $S\setminus \car(S)$ is a
manifold itself because the characteristic set $\car(S)$ is a closed
subset of $S$. We observe that both $F_1$ and
$F_2$ remain unchanged if the function $u$ defining the surface $S$ is
multiplied by a positive function, whereas $F_1$ changes sign and
$F_2$ remains unchanged if $u$ is
multiplied by a negative function. Since the zero set of the twice
differentiable submersion defining $S$ needs to remain unchanged, these
are the only two options which can occur. Observe that the vector field
$F_1$ on $S\setminus \car(S)$
is nothing but the vector field $\widehat{X}_S$ defined
in~(\ref{defn:E1}).

Recalling that $g_\eps$ is the restriction to
the surface $S$ of the Riemannian
metric on $M$ obtained by requiring $(X_1,X_2,\sqrt{\eps} X_0)$ to be a
global orthonormal frame, we further obtain
\begin{equation*}
  g_\eps(F_1,F_2)=0
\end{equation*}
as well as
\begin{equation}\label{eq:normF}
  g_\eps(F_1,F_1)=1
  \quad\mbox{and}\quad
  g_\eps(F_2,F_2)=\frac{(X_0u)^2}{(X_1u)^2+(X_2u)^2}+\frac{1}{\eps}\;.
\end{equation}
Thus, $(F_1,F_2)$ is an orthogonal frame for $T (S\setminus\car(S))$
for each Riemannian manifold $(S,g_\eps)$. While in general,
the frame $(F_1,F_2)$ is not orthonormal it has the nice property that it does
not depend on $\eps>0$, which aids the analysis of the convergence
of the operators $\Delta_\eps$ in the limit $\eps\to 0$. Since 
$F_1$ and $F_2$ are vector fields on $S\setminus\car(S)$, there exist
functions $b_1,b_2\colon S\setminus \car(S)\to\R$, not depending on
$\eps>0$, such that
\begin{equation}\label{F_structconst}
  [F_1,F_2]=b_1F_1+b_2F_2\;.
\end{equation}
Whereas determining the functions $b_1$ and $b_2$ explicitly
from~(\ref{defn:F1}) and
(\ref{defn:F2}) is a painful task, we can express them nicely in
terms of, following the notations in~\cite{barilarikohli},
the characteristic deviation $h$ and a
tensor $\eta$ related to the torsion.
Let $J\colon D\to D$ be the linear transformation induced by the
contact form $\omega$ by requiring that, for vector fields $X$ and $Y$ in
the distribution $D$,
\begin{equation}\label{defn:Jcontact}
  g\left(X,J(Y)\right)=\db\omega(X,Y)\;.
\end{equation}
Under the assumption of the existence of the global orthonormal frame
$(X_1,X_2)$ this amounts to saying that
\begin{equation}
  J(X_1)=X_2\label{defn:Jonframe}
  \quad\mbox{and}\quad
  J(X_2)=-X_1\;.
\end{equation}
For a unit-length vector field $X$ in the distribution $D$,
we use $[X,J(X)]|_D$ to denote the restriction of the vector field
$[X,J(X)]$ on $M$ to the distribution $D$ and we set
\begin{align*}
  h(X)&=-g\left([X,J(X)]|_D,X\right)\;,\\
  \eta(X)&=-g\left([X_0,X],X\right)\;,
\end{align*}
where the expression for $\eta$ is indeed well-defined because
according to~(\ref{defn:c01}) and (\ref{defn:c02}), the vector field
$[X_0,X]$ lies in the distribution $D$.
\begin{lemma}\label{lem:Fstruct}
  For $\drift\colon S\setminus\car(S)\to \R$ defined
  by~(\ref{defn:drift}), we have
  \begin{equation*}
    [F_1,F_2]=-\left(bh(F_1)+\eta(F_1)\right)F_1-bF_2\;,
  \end{equation*}
  that is, $b_1=-bh(F_1)-\eta(F_1)$ and $b_2=-b$.
\end{lemma}
\begin{proof}
  We first observe that due to~(\ref{defn:Jonframe}), we can write
  \begin{equation*}
    F_2=bJ(F_1)-X_0\;.
  \end{equation*}
  Using~(\ref{defn:c01}) and (\ref{defn:c02}) as well
  as~(\ref{defn:Jcontact}), it follows that
  \begin{equation*}
    \omega\left([F_1,F_2]\right)
    =\omega\left([F_1,bJ(F_1)-X_0]\right)
    =-\db\omega(F_1,bJ(F_1))
    =-g(F_1,bJ^2(F_1))=b\;.
  \end{equation*}
  On the other hand, from~(\ref{defn:F1}), (\ref{defn:F2}) and
  (\ref{F_structconst}), we deduce
  \begin{equation*}
    \omega\left([F_1,F_2]\right)
    =\omega\left(b_2F_2\right)=-b_2\;,
  \end{equation*}
  which implies that $b_2=-b$, as claimed. It remains to determine
  $b_1$. From~(\ref{defn:Jcontact}), we see that
  \begin{equation*}
    g(F_1,J(F_1))=-\omega([F_1,F_1])=0\;.
  \end{equation*}
  Together with~(\ref{F_structconst}) this yields
  \begin{equation*}
    b_1=g\left([F_1,F_2],F_1\right)
    =g\left([F_1,bJ(F_1)-X_0],F_1\right)
    =bg\left([F_1,J(F_1)]|_D,F_1\right)+g\left([X_0,F_1],F_1\right)\;,
  \end{equation*}
  and therefore, we have $b_1=-bh(F_1)-\eta(F_1)$, as required.
\end{proof}
To derive an expression for the Laplace--Beltrami operators
$\Delta_\eps$ of $(S,g_\eps)$ restricted to $S\setminus\car(S)$ in
terms of the vector fields $F_1$ and $F_2$, it is helpful to consider
the normalised frame associated with the orthogonal frame $(F_1,F_2)$.
For $\eps>0$ fixed, we define $a_\eps\colon S\setminus\car(S)\to \R$
by
\begin{equation}\label{defn:aeps}
  a_\eps=\left(
    \frac{(X_0u)^2}{(X_1u)^2+(X_2u)^2}+\frac{1}{\eps}
    \right)^{-\frac{1}{2}}
\end{equation}
and we introduce the vector fields $E_1$ and $E_{2,\eps}$ on
$S\setminus \car(S)$ given by
\begin{equation}\label{defn:E2}
  E_1=F_1\quad\mbox{and}\quad E_{2,\eps}=a_\eps F_2\;.
\end{equation}
In the Riemannian manifold $(S,g_\eps)$, this yields the orthonormal
frame $(E_1,E_{2,\eps})$ for $T (S\setminus\car(S))$.
\begin{lemma}\label{lem:LB_eps}
  For $\eps>0$, the operator $\Delta_\eps$
  restricted to $S\setminus\car(S)$ can be expressed as
  \begin{equation*}
    \Delta_\eps|_{S\setminus\car(S)}
    =F_1^2+a_\eps^2 F_2^2+
     \left(b-\frac{F_1(a_\eps)}{a_\eps}\right)F_1
     -a_\eps^2\left(bh(F_1)+\eta(F_1)\right)F_2\;.
  \end{equation*}
\end{lemma}
\begin{proof}
  Fix $\eps>0$ and let $\operatorname{div}_\eps$ denote the divergence
  operator on the Riemannian manifold $(S,g_\eps)$ with respect to the
  corresponding Riemannian volume form. Since
  $(E_1,E_{2,\eps})$ is an orthonormal frame for
  $T(S\setminus\car(S))$, we have
  \begin{equation}\label{LB_in_ONF}
    \Delta_\eps|_{S\setminus\car(S)}=
    E_1^2+E_{2,\eps}^2+
    \left(\operatorname{div}_\eps E_1\right)E_1+
    \left(\operatorname{div}_\eps E_{2,\eps}\right)E_{2,\eps}\;.
  \end{equation}
  Let $(\nu_1,\nu_{2,\eps})$ denote the dual to the orthonormal frame
  $(E_1,E_{2,\eps})$.
  Proceeding, for instance, in the same way as
  in~\cite[Proof of Proposition~11]{barilari},
  we show that, for any vector field $X$ on $S\setminus\car(S)$,
  \begin{equation*}
    \operatorname{div}_\eps X=\nu_1\left([E_1,X]\right) +
      \nu_{2,\eps}\left([E_{2,\eps},X]\right)\;.
  \end{equation*}
  This together with~(\ref{defn:E2}) and Lemma~\ref{lem:Fstruct}
  implies that
  \begin{equation*}
    \operatorname{div}_\eps E_1
    = \nu_{2,\eps}\left([a_\eps F_2,F_1]\right)
    = -\nu_{2,\eps}\left(a_\eps [F_1,F_2]+F_1(a_\eps)F_2\right)
    = b-\frac{F_1(a_\eps)}{a_\eps}
  \end{equation*}
  as well as
  \begin{equation*}
    \operatorname{div}_\eps E_{2,\eps}
    = \nu_1\left([F_1,a_\eps F_2]\right)
    = \nu_1\left(a_\eps[F_1, F_2] + F_1(a_\eps) F_2\right)
    = -a_\eps\left(bh(F_1)+\eta(F_1)\right)\;.
  \end{equation*}
  The desired result follows from~(\ref{defn:E2}) and (\ref{LB_in_ONF}).
\end{proof}
Note that $\Delta_\eps|_{S\setminus\car(S)}$ in Lemma~\ref{lem:LB_eps}
can equivalently be written as
\begin{equation*}
  \Delta_\eps|_{S\setminus\car(S)}
  =F_1^2+a_\eps^2 F_2^2+
  \left(b-\frac{F_1\left(a_\eps^2\right)}{2a_\eps^2}\right)F_1
     -a_\eps^2\left(bh(F_1)+\eta(F_1)\right)F_2\;.
\end{equation*}
Using Lemma~\ref{lem:LB_eps} we can prove
Theorem~\ref{thm:conv_of_LB}.
\begin{proof}[Proof of Theorem~\ref{thm:conv_of_LB}]
  From~(\ref{defn:drift}) and (\ref{defn:aeps}), we obtain that
  \begin{equation}\label{expr4ae}
    a_\eps^2=\left(\drift^2+\frac{1}{\eps}\right)^{-1}
    =\frac{\eps}{\eps \drift^2 +1}\;,
  \end{equation}
  which we use to compute
  \begin{equation*}
    \frac{F_1(a_\eps)}{a_\eps}
    =\frac{F_1(a_\eps^2)}{2a_\eps^2}
    =-\frac{\eps \drift F_1(\drift)}{\eps \drift^2+1}\;.    
  \end{equation*}
  It follows that
  \begin{equation}\label{basicbds}
    a_\eps^2\leq\eps
    \quad\mbox{as well as}\quad
    \left|\frac{F_1(a_\eps)}{a_\eps}\right|
    \leq \eps \left|\drift F_1(\drift)\right|\;.
  \end{equation}
  Since $u\in C^2(M)$ by assumption, both
  $\drift\colon S\setminus\car(S)\to \R$
  and $F_1(\drift)\colon S\setminus\car(S)\to \R$ are continuous and
  therefore bounded on compact subsets of
  $S\setminus\car(S)$. In a similar way, we argue that
  the function $b_1=-bh(F_1)-\eta(F_1)$ is
  bounded on compact subsets of $S\setminus\car(S)$.
  Due to~(\ref{basicbds}), this implies
  that, uniformly on compact subsets of $S\setminus\car(S)$,
  \begin{equation}\label{11conv}
    \lim_{\eps\to 0}a_\eps^2=0\;,\quad
    \lim_{\eps\to 0}\frac{F_1(a_\eps)}{a_\eps}=0
    \quad\mbox{and}\quad
    \lim_{\eps\to 0}a_\eps^2\left(bh(F_1)+\eta(F_1)\right)=0\;.
  \end{equation}
  Let $f\in C_c^2(S\setminus \car(S))$. We then have
  $F_1f,F_2f\in C_c^1(S\setminus \car(S))$ and
  $F_1^2f,F_2^2f\in C_c^0(S\setminus \car(S))$. Since the
  expression~(\ref{limitLB}) for $\Delta_0$ can be rewritten as
  \begin{equation*}
    \Delta_0=F_1^2+\drift F_1
  \end{equation*}
  and since the
  convergence in~(\ref{11conv}) is uniformly on compact subsets of
  $S\setminus\car(S)$, we deduce from Lemma~\ref{lem:LB_eps} that
  \begin{equation*}
    \lim_{\eps\to 0}
    \left\|\Delta_\eps f-\Delta_0 f\right\|_{\infty,S\setminus \car(S)}
    =\lim_{\eps\to 0}
    \left\|a_\eps^2 F_2^2f-
      \frac{F_1\left(a_\eps\right)}{a_\eps}F_1f
     -a_\eps^2\left(bh(F_1)+\eta(F_1)\right)F_2f
   \right\|_{\infty,S\setminus\car(S)} = 0\;,
  \end{equation*}
  that is, the functions $\Delta_\eps f$ indeed converge uniformly on
  $S\setminus\car(S)$ to $\Delta_0 f$.
\end{proof}
Using the orthonormal frames $(E_1,E_{2,\eps})$,
we easily derive the expression given in
Proposition~\ref{prop:gausscurv} for the
intrinsic Gaussian curvature $K_0$ of the surface $S$
in terms of the vector field $\widehat{X}_S$ and the
function $\drift$. Unlike the reasoning presented in~\cite{balogh},
which further exploits intrinsic
symmetries of the Heisenberg group $\Hs$, our derivation does not rely on
the cancellation of divergent quantities and holds for surfaces
in any three-dimensional contact sub-Riemannian manifold,
cf.~\cite[Remark~5.3]{balogh}.
\begin{proof}[Proof of Proposition~\ref{prop:gausscurv}]
  From Lemma~\ref{lem:Fstruct} and due to~(\ref{F_structconst}) as
  well as (\ref{defn:E2}), we have
  \begin{equation*}
    [E_1,E_{2,\eps}]
    =[F_1,a_\eps F_2]
    =a_\eps[F_1,F_2]+F_1(a_\eps)F_2
    =a_\eps b_1E_1+\left(-b + \frac{F_1(a_\eps)}{a_\eps}\right)E_{2,\eps}\;.
  \end{equation*}
  According to the classical formula for the Gaussian curvature of a
  surface in terms of an orthonormal frame, see
  e.g.~\cite[Proposition 4.40]{ABB}, the Gaussian curvature $K_\eps$ of
  the Riemannian manifold $(S,g_{\eps})$ is given by
  \begin{equation}\label{Gauss_curv}
    K_\eps=
    F_1\left(-\drift + \frac{F_1(a_\eps)}{a_\eps}\right)-
    a_\eps F_2\left(a_\eps b_1\right)-
    \left(a_\eps b_1\right)^2-
    \left(-\drift + \frac{F_1(a_\eps)}{a_\eps}\right)^2\;.
  \end{equation}
  We deduce from~(\ref{expr4ae}) that
  \begin{equation*}
    a_\eps F_2\left(a_\eps\right)
    =\frac{1}{2}F_2\left(a_\eps^2\right)
    =-\frac{\eps^2 \drift F_2(\drift)}{\left(\eps \drift^2+1\right)^2}
  \end{equation*}
  as well as
  \begin{equation*}
    F_1\left(\frac{F_1(a_\eps)}{a_\eps}\right)
    =-F_1\left(\frac{\eps \drift F_1(\drift)}{\eps \drift^2+1}\right)
    =-\frac{\eps F_1\left(\drift F_1(\drift)\right)}{\eps \drift^2+1}+
    \frac{2\eps^2\drift^2\left(F_1(\drift)\right)^2}
    {\left(\eps \drift^2+1\right)^2}\;,
  \end{equation*}
  which, in addition to~(\ref{basicbds}), implies
  \begin{equation*}
    \left|a_\eps F_2\left(a_\eps\right)\right|
    \leq\eps^2\left|\drift F_2(\drift)\right|
    \quad\mbox{and}\quad
    \left|F_1\left(\frac{F_1(a_\eps)}{a_\eps}\right)\right|
    \leq\eps\left|F_1\left(\drift F_1(\drift)\right)\right|+
    2\eps^2\drift^2\left(F_1(\drift)\right)^2\;.
  \end{equation*}
  By passing to the limit $\eps \to 0$ in~(\ref{Gauss_curv}), the
  desired expression follows.
\end{proof}
Notice that, by construction, the function $b$ and the intrinsic
Gaussian curvature $K_0$ are related by the Riccati-like equation
\begin{equation*}
  \dot b +b^{2}+K_{0}=0\;,
\end{equation*}
with the notation $\dot b=\widehat{X}_{S}(b)$.

\section{Canonical stochastic process on the embedded
  surface}
\label{sec:stochastic}
We study the stochastic process with generator $\frac{1}{2}\Delta_0$
on $S\setminus\car(S)$. After analysing the behaviour
of the drift of the 
process around non-degenerate characteristic points, we
prove Theorem~\ref{thm:alive} and Proposition~\ref{propn:hyperbolic}.

By construction, the process with generator $\frac{1}{2}\Delta_0$
moves along the characteristic foliation of $S$, that is, along the
integral curves of the vector field $\widehat{X}_S$ on
$S\setminus\car(S)$ defined in~(\ref{defn:E1}).
Around a fixed non-degenerate characteristic point $x\in\car(S)$, the
behaviour of the canonical stochastic process is determined by how
$\drift\colon S\setminus\car(S)\to \R$ given in
(\ref{defn:drift}) depends on the arc length along integral curves
emanating from $x$.
Since the vector fields $X_1,X_2$ and the Reeb vector field $X_0$ are
linearly independent everywhere, the function
$X_0u\colon S\to\R$ does not vanish near characteristic points. In
particular, we may and do choose the function $u\in C^2(M)$ defining the
surface $S$ such that $X_0u\equiv 1$ in a neighbourhood of $x$.

Understanding the expression for the horizontal Hessian $\hess u$
in~(\ref{defn:hess}) as a matrix representation in the dual frame of
$(X_1,X_2)$, and noting that the linear transformation
$J\colon D\to D$ defined in~(\ref{defn:Jcontact}) has the matrix
representation
\begin{equation*}
  J=
  \begin{pmatrix}
    0 & -1\\
    1 & 0
  \end{pmatrix}\;,
\end{equation*}
we see that
\begin{equation*}
  \left(\hess u\right)J=
  \begin{pmatrix}
    X_1X_2u & -X_1X_1u \\
    X_2X_2u & -X_2X_1u
  \end{pmatrix}\;.
\end{equation*}
The dynamics around
the characteristic point $x\in\car(S)$ is uniquely determined by the
eigenvalues
$\lambda_1$ and $\lambda_2$ of
$((\hess u)(x))J$. Since $x\in\car(S)$ is non-degenerate by assumption
both eigenvalues are non-zero, and due to $X_0u\equiv 1$ in a
neighbourhood of $x$, we further have
\begin{equation}\label{trace}
  \lambda_1+\lambda_2
  =\operatorname{Tr}\left(((\hess u)(x))J\right)
  =\left(X_1X_2u\right)(x)-\left(X_2X_1u\right)(x)
  =\left(X_0 u\right)(x) = 1\;.
\end{equation}
Thus, one of the following three cases occurs, where we use the
terminology from~\cite[Section~4.4]{robinson} to distinguish between
them. In the first case, where the eigenvalues $\lambda_1$ and
$\lambda_2$ are complex conjugate, the characteristic point $x$ is of
focus type and the
integral curves of $\widehat{X}_S$ spiral towards
the point $x$.
In the second case, where both eigenvalues are real and of positive
sign, we call $x\in\car(S)$ of node type, and all integral curves of
$\widehat{X}_S$ approaching $x$ do so tangentially to the
eigendirection corresponding to the smaller eigenvalue, with the
exception of the separatrices of the larger eigenvalue. In the third
case with the characteristic point $x$ being of saddle type, the two
eigenvalues are real but of opposite sign, and the only integral
curves of $\widehat{X}_S$ approaching $x$ are the separatrices.

Note that an elliptic characteristic point is of focus type or of node
type, whereas a hyperbolic characteristic point is of saddle
type. Depending on which of theses cases arises, we can determine how
the function $\drift$ depends on the arc
length along integral curves of $\widehat{X}_S$ emanating from
$x$.
The choice of the function $u\in C^2(M)$ such that $X_0u\equiv 1$ in a
neighbourhood of $x$ fixes the sign of the vector field
$\widehat{X}_S$. In particular, an integral curve $\gamma$ of
$\widehat{X}_S$ which extends continuously to $\gamma(0)=x$
might be defined either on the interval $[0,\delta)$ or on
$(-\delta,0]$ for some $\delta>0$. As the derivation presented below
works irrespective of the sign of the parameter of $\gamma$,
we combine the two
cases by writing $\gamma\colon I_\delta\to S$ for integral curves
of $\widehat{X}_S$ extended continuously to $\gamma(0)=x$.

The expansion around a characteristic point of focus type is a
result of the fact that the real parts of complex conjugate eigenvalues
satisfying~(\ref{trace}) equal $\frac{1}{2}$.
\begin{lemma}\label{lem:b4focus}
  Let $x\in\car(S)$ be a non-degenerate characteristic point and
  suppose that $u\in C^2(M)$ is chosen such that $X_0u\equiv 1$ in a
  neighbourhood of $x$. For $\delta>0$, let $\gamma\colon I_\delta\to S$
  be an integral curve of the vector field $\widehat{X}_S$
  extended continuously to $\gamma(0)=x$. If the
  eigenvalues of $((\hess u)(x))J$ are complex conjugate then, as
  $s\to 0$,
  \begin{equation*}
    \drift(\gamma(s))= \frac{2}{s}+O(1)\;.
  \end{equation*}
\end{lemma}
\begin{proof}
  Since $X_0u\equiv 1$ in a neighbourhood of $x$, we may suppose that
  $\delta>0$ is chosen small enough such that, for
  $s\in I_\delta\setminus\{0\}$,
  \begin{equation*}
    \drift(\gamma(s))=
      \frac{1}{\sqrt{\left(
            \left(X_1u\right)\left(\gamma(s)\right)\right)^2+
        \left(\left(X_2u\right)\left(\gamma(s)\right)\right)^2}}\;.
  \end{equation*}
  A direct computation shows
  \begin{equation*}
    \frac{\pt}{\pt s}\left(\drift(\gamma(s))^{-1}\right)
    =\widehat{X}_S\left(\drift(\gamma(s))^{-1}\right)
    =\left(\left(\hess u\right)\left(\gamma(s)\right)\right)
    \left(J\left(\widehat{X}_S\left(\gamma(s)\right)\right),
    \widehat{X}_S\left(\gamma(s)\right)\right)\;.
  \end{equation*}
  By the Hartman--Grobman theorem, it follows that,
  for $s\to 0$,
  \begin{equation*}
    \frac{\pt}{\pt s}\left(\drift(\gamma(s))^{-1}\right)
    =\left((\hess u)(x)\right)
    \left(J\left(\widehat{X}_S\left(\gamma(s)\right)\right),
    \widehat{X}_S\left(\gamma(s)\right)\right)+O(s)\;.
  \end{equation*}
  As complex conjugate eigenvalues of
  $((\hess u)(x))J$ have real part equal to $\frac{1}{2}$ and
  due to $\widehat{X}_S$ being a unit-length vector field, the
  previous expression simplifies to
  \begin{equation}\label{cc:taylor1}
    \frac{\pt}{\pt s}\left(\drift(\gamma(s))^{-1}\right)
    =\frac{1}{2}+O(s)\;.
  \end{equation}
  Since $(X_1u)(x)=(X_2u)(x)=0$ at the characteristic point $x$,
  we further have
  \begin{equation}\label{cc:taylor0}
    \lim_{s\to 0}\frac{1}{\drift(\gamma(s))}=0\;.
  \end{equation}
  A Taylor expansion together with~(\ref{cc:taylor1}) and (\ref{cc:taylor0})
  then implies that, as $s\to 0$,
  \begin{equation*}
    \frac{1}{\drift(\gamma(s))}
    =\frac{s}{2}+O\left(s^2\right)\;,
  \end{equation*}
  which yields, for $s\to 0$,
  \begin{equation*}
    \drift(\gamma(s))
    =\frac{2}{s}\left(1+O(s)\right)^{-1}
    =\frac{2}{s}+O(1)\;,
  \end{equation*}
  as claimed.
\end{proof}
The expansion of the function $b$ around characteristic points of node
type or of saddle type depends on along which integral curve of
$\widehat{X}_S$ we are expanding. By the discussions preceding
Lemma~\ref{lem:b4focus}, all possible behaviours are covered by the
next result.
\begin{lemma}\label{lem:b4rest}
  Fix a non-degenerate characteristic point $x\in\car(S)$.
  For $\delta>0$, let $\gamma\colon I_\delta\to S$
  be an integral curve of the vector field $\widehat{X}_S$
  which extends continuously to $\gamma(0)=x$.
  Assume $u\in C^2(M)$ is chosen such that $X_0u\equiv 1$ in a
  neighbourhood of $x$ and suppose $((\hess u)(x))J$ has real
  eigenvalues. If the curve $\gamma$ approaches $x$
  tangentially to the eigendirection corresponding to the eigenvalue
  $\lambda_i$, for $i\in\{1,2\}$, then, as $s\to 0$,
  \begin{equation*}
    \drift(\gamma(s))= \frac{1}{\lambda_i s}+O(1)\;.
  \end{equation*}
\end{lemma}
\begin{proof}
  As in the proof of Lemma~\ref{lem:b4focus}, we obtain,
  for $\delta>0$ small enough and $s\in I_\delta\setminus\{0\}$,
  \begin{equation*}
    \widehat{X}_S\left(\drift(\gamma(s))^{-1}\right)
    =\left(\left(\hess u\right)\left(\gamma(s)\right)\right)
    \left(J\left(\widehat{X}_S\left(\gamma(s)\right)\right),
    \widehat{X}_S\left(\gamma(s)\right)\right)\;.
  \end{equation*}
  Since $\gamma$ is an integral curve of the vector field
  $\widehat{X}_S$, we deduce that
  \begin{equation*}
    \frac{\pt}{\pt s}\left(\frac{1}{\drift(\gamma(s))}\right)
    =\left(\left(\hess u\right)\left(\gamma(s)\right)\right)
    \left(J\left(\gamma'(s)\right),\gamma'(s)\right)\;.
  \end{equation*}
  By Taylor expansion, this together with~(\ref{cc:taylor0}) yields, for
  $s\to 0$,
  \begin{equation*}
    \frac{1}{\drift(\gamma(s))}
    =\left(\left(\hess u\right)(x)\right)
    \left(J\left(\gamma'(0)\right),\gamma'(0)\right)s
    +O\left(s^2\right)\;.
  \end{equation*}
  By assumption, the vector $\gamma'(0)\in T_xS$ is
  a unit-length eigenvector of $((\hess u)(x))J$
  corresponding to the eigenvalue $\lambda_i$, which has to be 
  non-zero because $x$ is a non-degenerate characteristic point.
  It follows that
  \begin{equation*}
    \left(\left(\hess u\right)(x)\right)
    \left(J\left(\gamma'(0)\right),\gamma'(0)\right)
    =\lambda_i\not= 0\;,
  \end{equation*}
  which implies, for $s\to 0$,
  \begin{equation*}
    \drift(\gamma(s))
    =\frac{1}{\lambda_i s}\left(1+O(s)\right)^{-1}
    =\frac{1}{\lambda_i s}+O(1)\;,
  \end{equation*}
  as required.
\end{proof}
\begin{remark}\label{rem:saddle}
  We stress Lemma~\ref{lem:b4rest} does not contradict the
  positivity of the function $\drift$ near the point $x$ ensured by the
  choice of $u\in C^2(M)$ such that $X_0u\equiv 1$ in neighbourhood of
  $x$. The derived expansion for $\drift$ simply implies that
  on the separatrices corresponding to the negative eigenvalue of a
  hyperbolic characteristic point, the vector field $\widehat{X}_S$
  points towards the characteristic point
  for that choice of $u$, that is, we have $s\in (-\delta,0)$. At the
  same time, we notice that
  \begin{equation*}
    \frac{\pt^2}{\pt s^2}+
    \drift\left(\gamma(s)\right)\frac{\pt}{\pt s}
  \end{equation*}
  remains invariant under a change from $s$ to $-s$. Therefore, in our
  analysis of the one-dimensional diffusion processes induced on
  integral curves of $\widehat{X}_S$, we may again assume that the
  integral curves are parameterised by a positive parameter.
\end{remark}
With the classification of
singular points for stochastic differential equations given by Cherny
and Engelbert in~\cite[Section~2.3]{engelbert}, the previous two
lemmas provide what is needed to prove
Theorem~\ref{thm:alive} and Proposition~\ref{propn:hyperbolic}. One
additional crucial observation is that for a characteristic point of
node type both eigenvalues of $\left((\hess u)(x)\right)J$ are positive
and less than one, whereas for a characteristic point of saddle type,
the positive eigenvalue is greater than one.
\begin{proof}[Proof of Theorem~\ref{thm:alive}]
  Fix an elliptic characteristic point $x\in\car(S)$. For
  $\delta>0$, let
  $\gamma\colon[0,\delta]\to S$ be an integral curve of the vector
  field $\widehat{X}_S$ extended continuously to
  $x=\lim_{s\downarrow 0}\gamma(s)$. Following Cherny and
  Engelbert~\cite[Section~2.3]{engelbert}, since the
  one-dimensional diffusion process on $\gamma$ induced by
  $\frac{1}{2}\Delta_0$ has unit diffusivity and drift
  equal to $\frac{1}{2}b$, we set
  \begin{equation}\label{defn:rho}
    \rho(t)=\exp\left(\int_t^\delta \drift(\gamma(s))\dd s\right)
    \quad\mbox{for }t\in (0,\delta]\;.
  \end{equation}
  If the characteristic point $x$ is of node type
  the real positive eigenvalues
  $\lambda_1$ and $\lambda_2$ of $((\hess u)(x))J$
  satisfy $0<\lambda_1,\lambda_2<1$ by~(\ref{trace}).
  As $x$ is of focus type or of
  node type by assumption, Lemma~\ref{lem:b4focus} and
  Lemma~\ref{lem:b4rest} establish the existence of
  some $\lambda\in\R$ with
  $0<\lambda<1$ such that, as $s\downarrow 0$,
  \begin{equation*}
    \drift(\gamma(s))
    =\frac{1}{\lambda s}+O(1)\;.
  \end{equation*}
  We deduce, for $\delta>0$ sufficiently small,
  \begin{equation*}
    \rho(t)=\exp\left(
    \int_t^\delta\left(\frac{1}{\lambda s}+
      O\left(1\right)\right)\dd s\right)
    =\exp\left(\frac{1}{\lambda}
      \ln\left(\frac{\delta}{t}\right)+O(\delta-t)\right)
    =\left(\frac{\delta}{t}\right)^{\frac{1}{\lambda}}
     \left(1+O(\delta-t)\right).
  \end{equation*}
  Due to $\frac{1}{\lambda}>1$, this implies that
  \begin{equation*}
    \int_0^\delta\rho(t)\dd t=\infty\;.
  \end{equation*}
  According to~\cite[Theorem~2.16 and Theorem~2.17]{engelbert}, it
  follows that the elliptic characteristic point $x$
  is an inaccessible boundary point for the one-dimensional diffusion
  processes induced on
  the integral curves of $\widehat{X}_S$ emanating from $x$.
  Since $x\in\car(S)$ was an arbitrary elliptic characteristic point,
  the claimed result follows.
\end{proof}
\begin{proof}[Proof of Proposition~\ref{propn:hyperbolic}]
  We consider the stochastic process
  with generator $\frac{1}{2}\Delta_0$ on $S\setminus\car(S)$
  near a hyperbolic point $x\in\car(S)$.
  Let $\gamma$ be one of the four separatrices of $x$
  parameterised by arc
  length $s\geq 0$ and such that $\gamma(0)=x$.
  Let $\lambda_1$ be the positive eigenvalue and $\lambda_2$ be the
  negative eigenvalue of $((\hess u)(x))J$. From the trace
  property~(\ref{trace}), we see that $\lambda_1>1$. By
  Lemma~\ref{lem:b4rest} and Remark~\ref{rem:saddle}, we have,
  for $i\in\{1,2\}$ and as $s\downarrow 0$,
  \begin{equation*}
    \drift(\gamma(s))
    =\frac{1}{\lambda_i s}+O(1)\;.
  \end{equation*}
  As in the previous proof, for $\delta>0$ sufficiently small and
  $\rho\colon (0,\delta]\to\R$ defined by~(\ref{defn:rho}), we have
  \begin{equation*}
    \rho(t)
    =\left(\frac{\delta}{t}\right)^{\frac{1}{\lambda_i}}
     \left(1+O(\delta-t)\right)\;.
  \end{equation*}
  However, this time, due to $\frac{1}{\lambda_i}<1$
  for $i\in\{1,2\}$, we obtain
  \begin{equation*}
    \int_0^\delta\rho(t)\dd t<\infty\;.
  \end{equation*}
  Using $\frac{1}{\lambda_1}>0$, we further compute that, on the
  separatrices corresponding to the positive eigenvalue,
  \begin{equation*}
    \int_0^\delta\frac{1+\frac{1}{2}|\drift(\gamma(t))|}{\rho(t)}\dd t
    =\int_0^\delta\frac{t^{\frac{1}{\lambda_1}-1}}
     {2\lambda_1\delta^{\frac{1}{\lambda_1}}}\left(1+O(t)\right)\dd t
    <\infty
  \end{equation*}
  and
  \begin{equation*}
    \int_0^\delta\frac{|\drift(\gamma(t))|}{2}\dd t=\infty\;.
  \end{equation*}
  On the separatrices corresponding to the negative
  eigenvalue, we have, due to $\frac{1}{\lambda_2}<0$,
  \begin{equation*}
    \int_0^\delta\frac{1+\frac{1}{2}|\drift(\gamma(t))|}{\rho(t)}\dd t
    =\int_0^\delta\frac{t^{\frac{1}{\lambda_2}-1}}
     {2\lambda_2\delta^{\frac{1}{\lambda_2}}}\left(1+O(t)\right)\dd t
    =\infty
  \end{equation*}
  as well as
  \begin{equation*}
    s(t)=\int_0^t\rho(s)\dd s
    =\frac{\lambda_2\delta^{\frac{1}{\lambda_2}}}{\lambda_2-1}
     t^{1-\frac{1}{\lambda_2}}\left(1+O(t)\right)
  \end{equation*}
  and
  \begin{equation*}
    \int_0^\delta\frac{1+\frac{1}{2}|\drift(\gamma(t))|}
    {\rho(t)}s(t)\dd t
    =\int_0^\delta
    \frac{1}{2\left(\lambda_2-1\right)}\left(1+O(t)\right)\dd t
    <\infty\;.
  \end{equation*}
  Hence, as a consequence
  of the criterions~\cite[Theorem~2.12 and Theorem~2.13]{engelbert},
  the hyperbolic
  characteristic point $x$ is reached with positive probability 
  by the one-dimensional diffusion processes induced on the
  separatrices. Thus, the canonical stochastic process started on the
  separatrices is killed in finite time with positive
  probability.
\end{proof}

\section{Stochastic processes on quadric surfaces
  in the Heisenberg group}
\label{sec:heisenberg}
Let $\Hs$ be the first Heisenberg group, that is, the Lie group
obtained by endowing $\R^3$ with the group law, expressed in Cartesian
coordinates,
\begin{equation*}
  (x_1,y_1,z_1)\ast(x_2,y_2,z_2)=
  \left(x_1+x_2,y_1+y_2,z_1+z_2+\frac{1}{2}\left(x_1y_2-x_2y_1\right)\right)\;.
\end{equation*}
On $\Hs$, we consider the two left-invariant vector fields
\begin{equation*}
  X=\frac{\pt}{\pt x}-\frac{y}{2}\frac{\pt}{\pt z}
  \qquad\mbox{and}\qquad
  Y=\frac{\pt}{\pt y}+\frac{x}{2}\frac{\pt}{\pt z}\;,
\end{equation*}
and the contact form
\begin{equation*}
  \omega=\db z - \frac{1}{2}\left(x\dd y-y\dd x\right)\;.
\end{equation*}
We note that the vector fields $X$ and $Y$ span the contact
distribution $D$ corresponding to $\omega$, that they are orthonormal
with respect to the smooth fibre inner product $g$ on $D$ given by
\begin{equation*}
  g_{(x,y,z)}=\db x\otimes\db x+\db y\otimes\db y\;,
\end{equation*}
and that
\begin{equation*}
  \db\omega|_D=-\db x\wedge \db y=-\operatorname{vol}_g\;.
\end{equation*}
Therefore, the Heisenberg group $\Hs$ understood as the
three-dimensional contact sub-Riemannian manifold $(\R^3,D,g)$ falls
into our setting, with $X_1=X$, $X_2=Y$ and the Reeb vector field
\begin{equation*}
  X_0=\frac{\pt}{\pt z}=[X_1,X_2]\;.
\end{equation*}
In Section~\ref{sec:para} and in Section~\ref{sec:sph}, we discuss
paraboloids and ellipsoids of revolution admitting one or two
characteristic points, respectively, which are elliptic and of focus
type. For these examples, the characteristic foliations can
be described by logarithmic spirals in $\R^2$ lifted to the
paraboloids and spirals between the poles on the ellipsoids, which are 
loxodromes, also called rhumb lines, on spheres.
The induced
stochastic processes are the Bessel process of order $3$
for the paraboloids and
Legendre-like processes
for the ellipsoids
moving along the leaves of the 
characteristic foliation. In Section~\ref{sec:hyperpara}, we consider
hyperbolic paraboloids where, depending on a parameter, the
unique characteristic point is either
of saddle type or of node type, and we analyse the induced
stochastic processes on the separatrices.

\subsection{Paraboloid of revolution}
\label{sec:para}
For $a\in\R$, let $S$ be the Euclidean paraboloid of revolution given
by the equation $z=a(x^2+y^2)$ for Cartesian coordinates $(x,y,z)$
in the Heisenberg group $\Hs$. This corresponds to the surface
given by~(\ref{defn:Swu}) with $u\colon\R^3\to\R$ defined as
\begin{equation*}
  u(x,y,z)=z-a\left(x^2+y^2\right)\;.
\end{equation*}
We compute
\begin{equation*}
  X_0u\equiv 1\;,\quad
  \left(X_1u\right)(x,y,z)=-2ax-\frac{y}{2}
  \quad\mbox{and}\quad
  \left(X_2u\right)(x,y,z)=-2ay+\frac{x}{2}\;,
\end{equation*}
which yields
\begin{equation}\label{eq:parasum}
  \left((X_1u)(x,y,z)\right)^2+\left((X_2u)(x,y,z)\right)^2
  =\frac{1}{4}\left(1+16a^2\right)\left(x^2+y^2\right)\;.
\end{equation}
Thus, the origin of $\R^3$ is the only characteristic point on the
paraboloid $S$. It is elliptic and of focus type because
$X_0u\equiv 1$ and
\begin{equation*}
  \left(\hess u\right)J\equiv
  \begin{pmatrix}
    \frac{1}{2} & 2a \\[0.4em]
    -2a & \frac{1}{2}
  \end{pmatrix}
\end{equation*}
has eigenvalues $\frac{1}{2}\pm 2a\im$.
On $S\setminus\car(S)$, the vector field $\widehat{X}_S$
defined by~(\ref{defn:E1}) can be expressed as
\begin{equation}\label{eq:heisenEincar}
  \widehat{X}_S=
  \frac{1}{\sqrt{\left(1+16a^2\right)\left(x^2+y^2\right)}}
  \left(\left(x-4ay\right)\frac{\pt}{\pt x}+
    \left(y+4ax\right)\frac{\pt}{\pt y}+
    2a\left(x^2+y^2\right)\frac{\pt}{\pt z}
  \right)\;.
\end{equation}
Changing to cylindrical coordinates $(r,\theta,z)$ for
$\R^3\setminus\{0\}$ with $r>0$, $\theta\in[0,2\pi)$, $z\in\R$ and
using
\begin{equation*}
  r\frac{\pt}{\pt r}=
  x\frac{\pt}{\pt x}+y\frac{\pt}{\pt y}
  \quad\mbox{as well as}\quad
  \frac{\pt}{\pt \theta}=
  -y\frac{\pt}{\pt x}+x\frac{\pt}{\pt y}\;,
\end{equation*}
the expression~(\ref{eq:heisenEincar}) for the vector field
$\widehat{X}_S$ simplifies to
\begin{equation*}
  \widehat{X}_S=
  \frac{1}{\sqrt{1+16a^2}}
  \left(\frac{\pt}{\pt r}+
    \frac{4a}{r}\frac{\pt}{\pt\theta}+2ar\frac{\pt}{\pt z}
  \right)\;.
\end{equation*}
From~(\ref{eq:parasum}), we further obtain that the function
$\drift\colon S\setminus\car(S)\to \R$
defined by~(\ref{defn:drift}) can be written as
\begin{equation*}
  b(r,\theta,z)=
  \frac{1}{\sqrt{1+16a^2}}\frac{2}{r}\;.
\end{equation*}

\paragraph{\it Characteristic foliation}
The characteristic foliation induced on the paraboloid $S$ of
revolution by the contact structure $D$ of the Heisenberg group $\Hs$ is
described through the integral curves of the vector field
$\widehat{X}_S$, cf. Figure~\ref{fig:spirals}. Its integral curves are
spirals emanating from the origin which can be indexed by
$\psi\in[0,2\pi)$ and parameterised by $s\in(0,\infty)$ as
follows 
\begin{equation}\label{log_spirals}
  s\mapsto \left(
    \frac{s}{\sqrt{1+16a^2}},
    4a\ln\left(\frac{s}{\sqrt{1+16a^2}}\right)+\psi,
    \frac{as^2}{1+16a^2}\right)\;.
\end{equation}
By construction, the vector field $\widehat{X}_S$ is a unit vector field with
respect to each metric induced on the surface $S$ from Riemannian
approximations of the Heisenberg group. In particular, it follows that
the parameter $s\in(0,\infty)$ describes the arc length along the
spirals~(\ref{log_spirals}).
\begin{figure}[h]
  \centering
  \includegraphics[height=230pt]{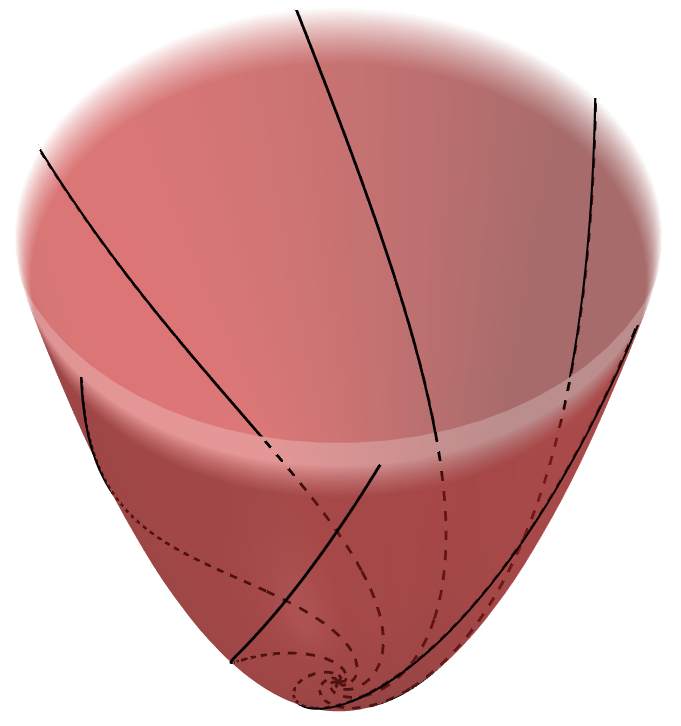}
  \caption{Characteristic foliation described by logarithmic spirals}
  \label{fig:spirals}
\end{figure}

\begin{remark}
  The spirals on $S$ defined by~(\ref{log_spirals}) are logarithmic
  spirals in $\R^2$ lifted to the paraboloid of revolution.
  In polar coordinates $(r,\theta)$ for $\R^2$, a logarithmic spiral can
  be written as
  \begin{equation}\label{log_spirals_R2}
    r=\e^{k\left(\theta+\theta_0\right)}
    \quad\mbox{for }k\in\R\setminus\{0\}
    \mbox{ and }\theta_0\in[0,2\pi)\;.
  \end{equation}
  Therefore, the spirals in~(\ref{log_spirals}) correspond to lifts of
  logarithmic spirals~(\ref{log_spirals_R2}) with $k=\frac{1}{4a}$.
  The arc length $s\in(0,\infty)$ of a logarithmic
  spiral~(\ref{log_spirals_R2}) measured from the origin satisfies
  \begin{equation*}
    s=\sqrt{1+\frac{1}{k^2}}\,r\;,
  \end{equation*}
  which for $k=\frac{1}{4a}$ yields
  $s=\sqrt{1+16a^2}\, r.$ Note that this is the same relation between
  arc length and radial distance as obtained for integral
  curves~(\ref{log_spirals}) of the vector field $\widehat{X}_S$. For
  further information on logarithmic spirals,
  see e.g. Zwikker~\cite[Chapter~16]{zwikker}.
\end{remark}
Using the spirals~(\ref{log_spirals}) which describe the characteristic
foliation on the paraboloid of revolution, we
introduce coordinates $(s,\psi)$ with $s>0$ and $\psi\in[0,2\pi)$ on the
surface $S\setminus\car(S)$. The vector field $\widehat{X}_S$ on
$S\setminus\car(S)$ and the function
$\drift\colon S\setminus\car(S)\to \R$ are then given by
\begin{equation*}
  \widehat{X}_S=\frac{\pt}{\pt s}
  \quad\mbox{and}\quad
  \drift(s,\psi)=\frac{2}{s}\;.
\end{equation*}
Thus, the canonical stochastic process induced on $S\setminus\car(S)$
has generator
\begin{equation*}
  \frac{1}{2}\Delta_0
  =\frac{1}{2}\left(\widehat{X}_S^2+\drift \widehat{X}_S\right)
  =\frac{1}{2}\frac{\pt^2}{\pt s^2}+\frac{1}{s}\frac{\pt}{\pt s}\;.
\end{equation*}
This gives rise to a Bessel process of order $3$ which out
of all the spirals~(\ref{log_spirals}) describing the characteristic
foliation on $S$ stays on the unique spiral passing through the
chosen starting point of the induced stochastic process.
In agreement with Theorem~\ref{thm:alive}, the origin is indeed
inaccessible for this stochastic process because a
Bessel process of order $3$ with positive starting point
remains positive almost surely. It arises as the radial component of a
three-dimensional Brownian
motion, and it is equal in law to a one-dimensional Brownian
motion started on the positive real line and conditioned to never hit
the origin.
We further observe that the operator $\Delta_0$ coincides with the
radial part of the Laplace--Beltrami operator for a quadratic cone,
cf.~\cite{Boscain-Neel,BP} for $\alpha=-2$, where the
self-adjointness of $\Delta_0$ is also studied.

As the limiting operator $\Delta_0$ does not depend on the parameter
$a\in\R$, the behaviour described above is also what we 
encounter on the plane $\{z=0\}$ in
the Heisenberg group $\Hs$, where the spirals~(\ref{log_spirals})
degenerate into rays emanating from the origin. We note that the
stochastic process induced by $\frac{1}{2}\Delta_0$ on the rays
differs from the singular diffusion introduced by Walsh~\cite{walsh}
on the same type of structure, but that it falls into the setting of
Chen and Fukushima~\cite{chen_onepoint}.

\subsection{Ellipsoid of revolution}
\label{sec:sph}
For $a,c\in\R$ positive,
we study the Euclidean spheroid, also called
ellipsoid of revolution, in the Heisenberg group $\Hs$ given by the
equation
\begin{equation*}
  \frac{x^2}{a^2}+\frac{y^2}{a^2}+\frac{z^2}{a^2c^2}=1
\end{equation*}
in Cartesian coordinates $(x,y,z)$.
To shorten the subsequent expressions, we choose
$u\colon \R^3\to\R$ defining the Euclidean spheroid $S$
through~(\ref{defn:Swu}) to be given by
\begin{equation*}
  u(x,y,z)=x^2+y^2+\frac{z^2}{c^2}-a^2\;.
\end{equation*}
Proceeding as in the previous example, we first obtain
\begin{equation*}
  \left(X_0u\right)(x,y,z)=\frac{2z}{c^2}
\end{equation*}
as well as
\begin{equation*}
  \left(X_1u\right)(x,y,z)
  =2x-\frac{yz}{c^2}
  \quad\mbox{and}\quad
  \left(X_2u\right)(x,y,z)
  =2y+\frac{xz}{c^2}\;,
\end{equation*}
which yields
\begin{equation}\label{ell:sqrt}
  \left((X_1u)(x,y,z)\right)^2+\left((X_2u)(x,y,z)\right)^2
  =\left(x^2+y^2\right)\left(4+\frac{z^2}{c^4}\right)\;.
\end{equation}
This implies the north pole $(0,0,ac)$ and the south pole
$(0,0,-ac)$ are the only two characteristic points on the
spheroid $S$. We further compute that
\begin{equation}\label{ell:XS}
  (X_2u)X_1-(X_1u)X_2
  =\left(2y+\frac{xz}{c^2}\right)\frac{\pt}{\pt x}
   -\left(2x-\frac{yz}{c^2}\right)\frac{\pt}{\pt y}
   -\left(x^2+y^2\right)\frac{\pt}{\pt z}\;.
\end{equation}
Using adapted
spheroidal coordinates $(\theta,\varphi)$ for
$S\setminus\car(S)$ with
$\theta\in(0,\pi)$ and $\varphi\in[0,2\pi)$, which are
related to the coordinates $(x,y,z)$ by
\begin{equation*}
  x=a\sin(\theta)\cos(\varphi)\;,\quad
  y=a\sin(\theta)\sin(\varphi)\;,\quad
  z=ac\cos(\theta)\;,
\end{equation*}
we have
\begin{align*}
  \frac{a\sin(\theta)}{c}\frac{\pt}{\pt\theta}
  &=\frac{xz}{c^2}\frac{\pt}{\pt x}
   +\frac{yz}{c^2}\frac{\pt}{\pt y}
   -\left(x^2+y^2\right)\frac{\pt}{\pt z}\quad\mbox{and}\\
  \frac{\pt}{\pt\varphi}
  &=-y\frac{\pt}{\pt x}+x\frac{\pt}{\pt y}\;.
\end{align*}
It follows that~(\ref{ell:XS}) on the surface $S\setminus\car(S)$
simplifies to
\begin{equation*}
  (X_2u)X_1-(X_1u)X_2
  =\frac{a\sin(\theta)}{c}\frac{\pt}{\pt\theta}
  -2\frac{\pt}{\pt\varphi}\;,
\end{equation*}
whereas~(\ref{ell:sqrt}) on $S\setminus\car(S)$ rewrites as
\begin{equation*}
  \left((X_1u)(\theta,\varphi)\right)^2+
  \left((X_2u)(\theta,\varphi)\right)^2
  =a^2\left(\sin(\theta)\right)^2
  \left(4+\frac{a^2\left(\cos(\theta)\right)^2}{c^2}\right)\;.
\end{equation*}
This shows that the vector field $\widehat{X}_S$ on
$S\setminus\car(S)$ defined by~(\ref{defn:E1}) is given as
\begin{equation}\label{ell:E1}
  \widehat{X}_S
  =\frac{1}{\sqrt{4c^2+a^2\left(\cos(\theta)\right)^2}}
  \left(\frac{\pt}{\pt\theta}
  -\frac{2c}{a\sin(\theta)}\frac{\pt}{\pt\varphi}\right)\;.
\end{equation}
For the function $\drift\colon S\setminus\car(S)\to \R$ defined
by~(\ref{defn:drift}), we further obtain that
\begin{equation}\label{ell:drift}
  b(\theta,\varphi)=
  \frac{2\cot(\theta)}{\sqrt{4c^2+a^2\left(\cos(\theta)\right)^2}}\;.
\end{equation}
As in the preceding example, in order to understand the canonical
stochastic process induced by the operator $\frac{1}{2}\Delta_0$ defined
through~(\ref{limitLB}), we need to express the vector field
$\widehat{X}_S$ and the function $\drift$ in terms of the arc length
along the integral curves of $\widehat{X}_S$. Since both
$\widehat{X}_S$ and $\drift$ are invariant under rotations along the
azimuthal angle $\varphi$,
this amounts to changing coordinates on the spheroid $S$ from
$(\theta,\varphi)$ to $(s,\varphi)$ where $s=s(\theta)$ is uniquely defined
by requiring that
\begin{equation*}
  \frac{\pt}{\pt s}=
  \frac{1}{\sqrt{4c^2+a^2\left(\cos(\theta)\right)^2}}
  \left(\frac{\pt}{\pt\theta}
  -\frac{2c}{a\sin(\theta)}\frac{\pt}{\pt\varphi}\right)
  \qquad\mbox{and}\qquad
  s(0)=0\;.
\end{equation*}
This corresponds to
\begin{equation}\label{dthetads}
  \frac{\db \theta}{\db s}=
  \frac{1}{\sqrt{4c^2+a^2\left(\cos(\theta)\right)^2}}\;,
\end{equation}
which together with $s(0)=0$ yields
\begin{equation*}
  s(\theta)=\int_0^\theta\sqrt{4c^2+a^2\left(\cos(\tau)\right)^2}\dd\tau
  =\int_0^\theta\sqrt{\left(4c^2+a^2\right)
    -a^2\left(\sin(\tau)\right)^2}\dd\tau
  \quad\mbox{for }\theta\in(0,\pi)\;.
\end{equation*}
Hence, the arc length $s$ along the integral curves of $\widehat{X}_S$
is given in terms of the polar angle $\theta$ as a multiple of an
elliptic integral of the second kind. Consequently, the
question if $\theta$ can be expressed explicitly in terms of $s$ is open.
However, for our analysis, it is sufficient that the map
$\theta\mapsto s(\theta)$ is invertible and that~(\ref{ell:drift}) as
well as (\ref{dthetads}) then imply
\begin{equation*}
  \drift(s,\varphi)=
  2\cot\left(\theta(s)\right)\frac{\db \theta}{\db s}\;.
\end{equation*}
Therefore, using the coordinates $(s,\varphi)$,
the operator $\frac{1}{2}\Delta_0$ on $S\setminus\car(S)$ can be
expressed as
\begin{equation*}
  \frac{1}{2}\Delta_0=\frac{1}{2}\frac{\pt^2}{\pt s^2}+
  \left(\cot\left(\theta(s)\right)\frac{\db \theta}{\db s}\right)
  \frac{\pt}{\pt s}\;,
\end{equation*}
which depends on the constants $a,c\in\R$
through~(\ref{dthetads}). Without the Jacobian factor
$\frac{\db \theta}{\db s}$ appearing in the drift term,
the canonical stochastic process induced by the operator
$\frac{1}{2}\Delta_0$ and moving along the leaves of the
characteristic foliation
would be a Legendre process, that is, a Brownian motion started inside
an interval and conditioned not to hit either endpoint of the
interval. The reason for the appearance of the additional factor
$\frac{\db \theta}{\db s}$ is that the integral curves of
$\widehat{X}_S$ connecting the two characteristic points are spirals
and not just great circles. For some further discussions on the
characteristic foliation of the spheroid, see the subsequent
Remark~\ref{rem:loxodromes}.

The emergence of an operator
which is almost the generator of
a Legendre process moving along the leaves of the
characteristic foliation motivates the 
search for a surface in a three-dimensional contact sub-Riemannian
manifold where we do exhibit a Legendre process moving along the
leaves of the
characteristic foliation induced by the contact structure. This is
achieved in Section~\ref{sec:SU}.

\begin{remark}
  The northern hemisphere of the spheroid could equally be defined by the
  function 
  \begin{equation*}
    u(x,y,z)=z-c\sqrt{a^2-x^2-y^2}\;.
  \end{equation*}
  With this choice we have $X_0u\equiv 1$. We further obtain
  \begin{equation*}
    \left(\left(\hess u\right)(0,0,ac)\right)J=
    \begin{pmatrix}
      \frac{1}{2} & -\frac{c}{a}\\[0.4em]
      \frac{c}{a} & \frac{1}{2}
    \end{pmatrix}\;,
  \end{equation*}
  whose eigenvalues are $\frac{1}{2}\pm\frac{c}{a}\im$. A similar
  computation on the southern hemisphere implies that both
  characteristic points are elliptic and of focus type. Thus, by
  Theorem~\ref{thm:alive}, the stochastic process with generator
  $\frac{1}{2}\Delta_0$ hits neither the north pole nor the south pole,
  and it induces a one-dimensional process on the unique leaf of
  the characteristic foliation picked out by the starting point.
\end{remark}
\begin{remark}\label{rem:loxodromes}
  With respect to the Euclidean
  metric $\langle\cdot,\cdot\rangle$ on $\R^3$, we have for the adapted
  spheroidal coordinates $(\theta,\varphi)$ of $S\setminus\car(S)$ as
  above that
  \begin{equation*}
    \left\langle\frac{\pt}{\pt\theta},
      \frac{\pt}{\pt\theta}\right\rangle
    =a^2\left(\cos(\theta)\right)^2 +
    a^2c^2\left(\sin(\theta)\right)^2
    \quad\mbox{and}\quad
    \left\langle\frac{\pt}{\pt\varphi},
      \frac{\pt}{\pt\varphi}\right\rangle
    =a^2\left(\sin(\theta)\right)^2\;.
  \end{equation*}
  It follows that the angle $\alpha$ formed by the
  vector field $\widehat{X}_S$ given in~(\ref{ell:E1}) and the azimuthal
  direction satisfies
  \begin{equation*}
    \cos\left(\alpha(\theta,\varphi)\right)=
    -\frac{2c}{\sqrt{a^2\left(\cos(\theta)\right)^2 +
        a^2c^2\left(\sin(\theta)\right)^2+4c^2}}\;.
  \end{equation*}
  Notably, on spheres, that is, if $c=1$, the angle $\alpha$ is
  constant everywhere. Hence, the integral curves of $\widehat{X}_S$
  considered as Euclidean curves on an
  Euclidean sphere are loxodromes, cf. Figure~\ref{fig:loxo},
  which are also called rhumb lines.
  They are related to logarithmic spirals through stereographic
  projection. Loxodromes arise in navigation by following a path with
  constant bearing measured with respect to the north pole or the
  south pole, see Carlton-Wippern~\cite{loxodrome}.
  \begin{figure}[h]
    \centering
    \includegraphics[height=220pt]{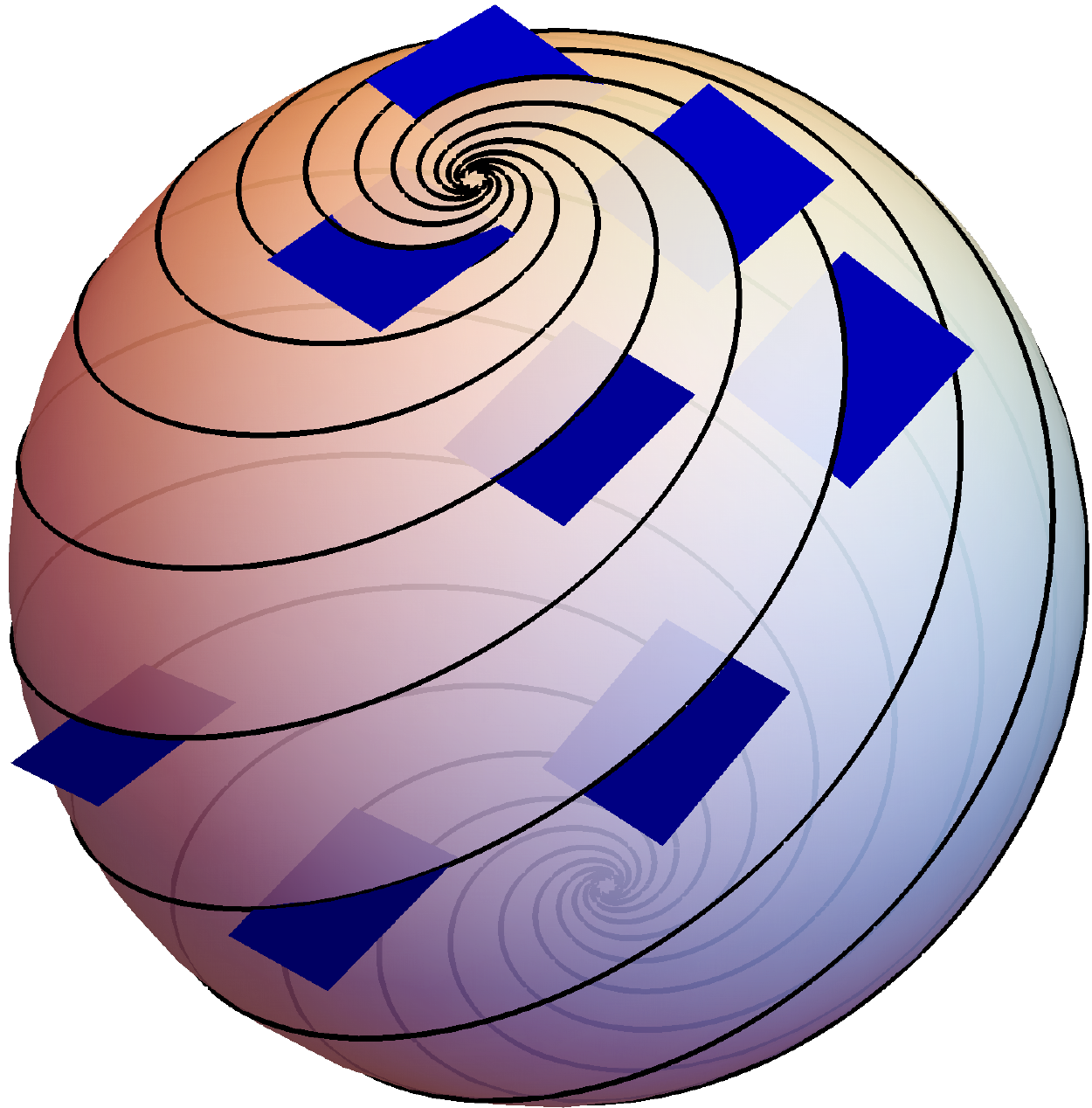}
    \caption{Characteristic foliation on spheres described by
      loxodromes}
    \label{fig:loxo}
  \end{figure}
\end{remark}

\subsection{Hyperbolic paraboloid}
\label{sec:hyperpara}
For $a\in\R$ positive and such that $a\not=\frac{1}{2}$, we consider
the Euclidean hyperbolic paraboloid $S$ in the Heisenberg group $\Hs$
given by~(\ref{defn:Swu}) with $u\colon\R^3\to\R$ defined as
\begin{equation*}
  u(x,y,z)=z-axy\;,
\end{equation*}
for Cartesian coordinates $(x,y,z)$. We compute
\begin{equation}\label{hyperpara:com}
  X_0u\equiv 1\;,\quad
  \left(X_1u\right)(x,y,z)=-ay-\frac{y}{2}
  \quad\mbox{as well as}\quad
  \left(X_2u\right)(x,y,z)=-ax+\frac{x}{2}\;,
\end{equation}
and further that
\begin{equation}\label{hyper:Hess}
  \left(\hess u\right)J\equiv
  \begin{pmatrix}
    \frac{1}{2}-a & 0 \\[0.4em]
    0 & \frac{1}{2}+a
  \end{pmatrix}\;.
\end{equation}
Due to
\begin{equation*}
  \left((X_1u)(x,y,z)\right)^2+\left((X_2u)(x,y,z)\right)^2
  =\left(\frac{1}{2}-a\right)^2x^2+
  \left(\frac{1}{2}+a\right)^2y^2\;,
\end{equation*}
the hyperbolic paraboloid $S$ has the origin of $\R^3$ as its unique
characteristic point. By~(\ref{hyper:Hess}), this characteristic
point is elliptic and of node type if $0 < a < \frac{1}{2}$, and
hyperbolic and therefore of saddle type if $a>\frac{1}{2}$. The reason
for having excluded the case $a=\frac{1}{2}$ right from the beginning is
that it gives rise to a line of degenerate characteristic points.

We note that the $x$-axis and the $y$-axis lie in the hyperbolic
paraboloid $S$. From~(\ref{hyperpara:com}), we see that the positive
and negative $x$-axis as well as the positive and negative
$y$-axis are integral curves of the vector field
$\widehat{X}_S$ on $S\setminus\car(S)$.
In the following, we restrict our attention to studying the behaviour
of the canonical stochastic process on these integral curves, which
nevertheless nicely illustrates Theorem~\ref{thm:alive} and
Proposition~\ref{propn:hyperbolic}.

We start by analysing the
one-dimensional diffusion process induced on the positive $y$-axis
$\gamma_y^+$, which by symmetry is equal in law to the
process induced on the negative $y$-axis. For all positive $a\in\R$
with $a\not=\frac{1}{2}$, we have
\begin{equation*}
  \widehat{X}_S|_{\gamma_y^+}=\frac{\pt}{\pt y}\;,
\end{equation*}
implying that the arc length $s>0$ along $\gamma_y^+$ is given by
$s=y$. This yields, for all $s>0$,
\begin{equation*}
  b\left(\gamma_y^+(s)\right)
  =\frac{1}{\left(\frac{1}{2}+a\right)s}\;.
\end{equation*}
Thus, the one-dimensional diffusion process on $\gamma_y^+$ induced by
$\frac{1}{2}\Delta_0$ has generator
\begin{equation*}
  \frac{1}{2}\frac{\pt^2}{\pt s^2}+
  \frac{1}{\left(1+2a\right)s}\frac{\pt}{\pt s}\;,
\end{equation*}
which gives rise to a Bessel process of order
$1+\frac{2}{1+2a}$. If started at a point with positive value this
diffusion process stays positive for all times
almost surely if $1+\frac{2}{1+2a}>2$
whereas it hits the origin with positive probability if
$1+\frac{2}{1+2a}<2$. This is consistent with Theorem~\ref{thm:alive}
and Proposition~\ref{propn:hyperbolic} because
for $a>\frac{1}{2}$ the positive
$y$-axis is a separatrix for the hyperbolic characteristic
point at the origin and
\begin{equation*}
  2<1+\frac{2}{1+2a}
  \quad\mbox{if }0<a<\frac{1}{2}
  \qquad\mbox{as well as}\qquad
  2>1+\frac{2}{1+2a}
  \quad\mbox{if }a>\frac{1}{2}\;.
\end{equation*}
Some more care is needed when studying the diffusion process induced on
the positive $x$-axis $\gamma_x^+$. As before, this process is
equal in law to the process induced on the negative $x$-axis.
We obtain
\begin{equation*}
  \widehat{X}_S|_{\gamma_x^+}=
  \begin{cases}
    \phantom{-}\frac{\pt}{\pt x} & \mbox{if } 0<a<\frac{1}{2}\\[0.4em]
    -\frac{\pt}{\pt x} & \mbox{if } a>\frac{1}{2}
  \end{cases}
\end{equation*}
as well as, for $x>0$,
\begin{equation*}
  b(x,0,0)=
  \begin{cases}
    \phantom{-}\dfrac{1}{\left(\frac{1}{2}-a\right)x}
    & \mbox{if } 0<a<\frac{1}{2}\\[1em]
    -\dfrac{1}{\left(\frac{1}{2}-a\right)x} & \mbox{if } a>\frac{1}{2}
  \end{cases}\;.
\end{equation*}
It follows that the one-dimensional diffusion process on $\gamma_x^+$
induced by $\frac{1}{2}\Delta_0$ has generator
\begin{equation*}
  \frac{1}{2}\frac{\pt^2}{\pt x^2}+
  \frac{1}{\left(1-2a\right)x}\frac{\pt}{\pt x}\;.
\end{equation*}
This yields a Bessel process of order $1+\frac{2}{1-2a}$. In
agreement with Theorem~\ref{thm:alive} and
Proposition~\ref{propn:hyperbolic},
if started at a point with positive value this process
never reaches the origin if $0<a<\frac{1}{2}$ which ensures
$1+\frac{2}{1-2a}>3$, whereas the process reaches the origin with
positive probability if $a>\frac{1}{2}$ as this corresponds to
$1+\frac{2}{1-2a}<1$.

\section{Stochastic processes on canonical surfaces in
  \texorpdfstring{$\SU$}{SU(2)} and
  \texorpdfstring{$\SL$}{SL(2,R)}}
\label{sec:model}
In Section~\ref{sec:para}, we establish that for a
paraboloid of revolution embedded in the Heisenberg group $\Hs$, the
operator $\frac{1}{2}\Delta_0$ induces a Bessel process of order $3$
moving along the leaves of the characteristic foliation, which is
described by lifts
of logarithmic spirals emanating from the origin. As discussed in
Revuz and Yor~\cite[Chapter~VIII.3]{revuz}, the Legendre processes and
the hyperbolic Bessel processes arise from the same type of Girsanov
transformation as the Bessel process, where these three cases only
differ by the sign of a parameter. We further recall that in
Section~\ref{sec:sph} we encounter a canonical stochastic process
which is
almost a Legendre process moving along the leaves of the
characteristic foliation 
induced on a spheroid in the Heisenberg group $\Hs$.
This motivates the search for surfaces in
three-dimensional contact sub-Riemannian manifolds where the canonical
stochastic process is a Legendre process of order $3$ or a hyperbolic Bessel
process of order $3$
moving along the leaves of the characteristic foliation.

We consider surfaces in the Lie groups $\SU$ and $\SL$ endowed
with standard sub-Riemannian structures. Together with the
Heisenberg group, these sub-Riemannian geometries play the role
of model spaces for three-dimensional contact sub-Riemannian
manifolds. In the first two subsections,
we find, by explicit computations, the canonical stochastic processes
induced on certain surfaces in these groups, when expressed in
convenient coordinates. The last
subsection proposes a unified geometric description, justifying the
choice of our surfaces.

\subsection{Special unitary group
\texorpdfstring{$\SU$}{SU(2)}}
\label{sec:SU}
One obstruction to recovering Legendre processes moving along
the characteristic foliation in Section~\ref{sec:sph} is that the
characteristic foliation of a spheroid in the Heisenberg group is
described by spirals connecting the north pole and the south pole
instead of great circles. This is the reason for
considering $S^2$ as a surface embedded in $\SU\simeq S^3$ understood
as a contact sub-Riemannian manifold because this gives rise to a
characteristic foliation on $S^2$ described by great circles.

The special unitary group $\SU$ is the Lie group of
$2\times 2$ unitary matrices of determinant $1$, that is,
\begin{equation*}
  \SU=\left\{
    \begin{pmatrix}
      \phantom{-}z+w\im  & y+x\im\\
      -y+x\im & z-w\im
    \end{pmatrix}
    \colon
    x,y,z,w\in\R\mbox{ with }x^2+y^2+z^2+w^2=1
  \right\}\;,
\end{equation*}
with the group operation being given by matrix multiplication.
Using the Pauli matrices
\begin{equation*}
  \sigma_1=
  \begin{pmatrix}
    0 & 1 \\ 1 & 0
  \end{pmatrix}\;,\quad
  \sigma_2=
  \begin{pmatrix}
    0 & -\im \\ \im & 0
  \end{pmatrix}
  \quad\mbox{and}\quad
  \sigma_3=
  \begin{pmatrix}
    1 & 0 \\ 0 & -1
  \end{pmatrix}\;,
\end{equation*}
we identify $\SU$
with the unit quaternions, and hence also with $S^3$, via the map
\begin{equation*}
  \begin{pmatrix}
      \phantom{-}z+w\im  & y+x\im\\
      -y+x\im & z-w\im
    \end{pmatrix}
    \mapsto z I_2+x\im\sigma_1+y\im\sigma_2+w\im\sigma_3\;.
\end{equation*}
The Lie algebra $\mathfrak{su}(2)$ of $\SU$ is the algebra formed by
the $2\times 2$ skew-Hermitian matrices with trace zero. A basis for
$\mathfrak{su}(2)$ is
$\{\frac{\im\sigma_1}{2},\frac{\im\sigma_2}{2},\frac{\im\sigma_3}{2}\}$
and the corresponding left-invariant vector fields on the Lie group
$\SU$ are
\begin{align*}
  U_1&=\frac{1}{2}\left(
    -x\frac{\pt}{\pt z} + z\frac{\pt}{\pt x}
    -w\frac{\pt}{\pt y} + y\frac{\pt}{\pt w}
  \right)\;,\\
  U_2&=\frac{1}{2}\left(
    -y\frac{\pt}{\pt z} + w\frac{\pt}{\pt x}
    +z\frac{\pt}{\pt y} - x\frac{\pt}{\pt w}
  \right)\;,\\
  U_3&=\frac{1}{2}\left(
    -w\frac{\pt}{\pt z} - y\frac{\pt}{\pt x}
    +x\frac{\pt}{\pt y} + z\frac{\pt}{\pt w}
  \right)\;,
\end{align*}
which satisfy the commutation
relations $[U_1,U_2]=-U_3$, $[U_2,U_3]=-U_1$ and
$[U_3,U_1]=-U_2$. Thus, any two of these three left-invariant
vector fields give rise to a sub-Riemannian structure on $\SU$.
To streamline the subsequent computations, we
choose $k\in\R$ with $k>0$ and
equip $\SU$ with the sub-Riemannian structure obtained by setting
$X_1=2kU_1$, $X_2=2kU_2$ and by requiring $(X_1,X_2)$ to be
an orthonormal frame for the distribution $D$ spanned by
the vector fields $X_1$ and
$X_2$. The appropriately normalised contact form $\omega$ for the
contact distribution $D$ is
\begin{equation*}
  \omega= \frac{1}{2k^2}
  \left(w\dd z + y\dd x - x\dd y - z\dd w\right)
\end{equation*}
and the associated Reeb vector field $X_0$ satisfies
\begin{equation*}
  X_0=[X_1,X_2]=-4k^2U_3
  =2k^2\left(
    w\frac{\pt}{\pt z} + y\frac{\pt}{\pt x}
    -x\frac{\pt}{\pt y} - z\frac{\pt}{\pt w}
  \right)\;.
\end{equation*}
In $\SU$, we consider the surface $S$ given by the function
$u\colon\SU\to\R$ defined by
\begin{equation*}
  u(x,y,z,w)=w\;.
\end{equation*}
The surface $S$ is isomorphic to $S^2$ because
\begin{equation*}
  S=\left\{
    \begin{pmatrix}
      z & y+x\im\\
      -y+x\im & z
    \end{pmatrix}
    \colon
    x,y,z\in\R\mbox{ with }x^2+y^2+z^2=1
  \right\}\;.
\end{equation*}
We compute
\begin{equation*}
  (X_0u)(x,y,z,w)=-2k^2z\;,\quad
  (X_1u)(x,y,z,w)=ky
  \quad\mbox{and}\quad
  (X_2u)(x,y,z,w)=-kx\;,
\end{equation*}
which yields
\begin{equation*}
  \left((X_1u)(x,y,z,w)\right)^2+
  \left((X_2u)(x,y,z,w)\right)^2
  =k^2\left(x^2+y^2\right)\;.
\end{equation*}
Due to $x^2+y^2+z^2=1$, it follows that a point on $S$ is
characteristic if and only if $z=\pm 1$.
Thus, the characteristic
points on $S$ are the north pole $(0,0,1)$ and the south pole
$(0,0,-1)$. The vector field $\widehat{X}_S$
on $S\setminus\car(S)$ defined by~(\ref{defn:E1}) is given as
\begin{equation}\label{SU:E1}
  \widehat{X}_S=\frac{k}{\sqrt{x^2+y^2}}\left(
    \left(x^2+y^2\right)\frac{\pt}{\pt z}
    -xz\frac{\pt}{\pt x}-yz\frac{\pt}{\pt y}\right)\;,
\end{equation}
and for the function $\drift\colon S\setminus\car(S)\to \R$ defined
by~(\ref{defn:drift}), we obtain
\begin{equation}\label{SU:drift}
  \drift(x,y,z)=-\frac{2kz}{\sqrt{x^2+y^2}}\;.
\end{equation}
We now change coordinates for $S\setminus\car(S)$ from $(x,y,z)$ with
$x^2+y^2+z^2=1$ and $z\not=\pm 1$ to $(\theta,\varphi)$
with $\theta\in(0,\frac{\pi}{k})$ and $\varphi\in[0,2\pi)$ by
\begin{equation*}
  x=\sin(k\theta)\cos(\varphi)\;,\quad
  y=\sin(k\theta)\sin(\varphi)\quad\mbox{and}\quad
  z=\cos(k\theta)\;.
\end{equation*}
We note that
\begin{equation*}
  \frac{\pt}{\pt\theta}
  =k\cos(k\theta)\cos(\varphi)\frac{\pt}{\pt x} +
   k\cos(k\theta)\sin(\varphi)\frac{\pt}{\pt y} -
   k\sin(k\theta)\frac{\pt}{\pt z}
\end{equation*}
as well as
\begin{equation*}
  xz=\sin(k\theta)\cos(k\theta)\cos(\varphi)\;,\quad
  yz=\sin(k\theta)\cos(k\theta)\sin(\varphi)
  \quad\mbox{and}\quad
  \sqrt{x^2+y^2}=\sin(k\theta)\;.
\end{equation*}
This together with~(\ref{SU:E1}) and~(\ref{SU:drift})
implies that
\begin{equation*}
  \widehat{X}_S=-\frac{\pt}{\pt\theta}
  \quad\mbox{and}\quad
  \drift(\theta,\varphi)=-2k\cot(k\theta)\;.
\end{equation*}
We deduce that the integral curves of $\widehat{X}_S$ are great
circles on $S$ and that
\begin{equation*}
  \frac{1}{2}\Delta_0=
  \frac{1}{2}\frac{\pt^2}{\pt \theta^2}+
  k\cot(k\theta)\frac{\pt}{\pt \theta}\;,
\end{equation*}
which indeed, on each great circle, induces
a Legendre process of order $3$ on the interval
$(0,\frac{\pi}{k})$. These processes first appeared in
Knight~\cite{knight} as so-called taboo processes and are obtained by
conditioning Brownian motion started inside the interval
$(0,\frac{\pi}{k})$ to never hit either of the two boundary points,
see Bougerol and Defosseux~\cite[Section~5.1]{legendre1}.
As discussed
in It\^{o} and McKean~\cite[Section~7.15]{legendre2},
they also arise as the latitude of a Brownian motion on the 
three-dimensional sphere of radius $\frac{1}{k}$.

\subsection{Special linear group
\texorpdfstring{$\SL$}{SL(2,R)}}
\label{sec:SL}
The appearance of the Bessel process on
the plane $\{z=0\}$ in the Heisenberg 
group $\Hs$ and of the Legendre processes on a compactified plane in
$\SU$ understood as a contact sub-Riemannian manifold suggests that the
hyperbolic Bessel processes arise on planes in the special linear
group $\SL$ equipped with a sub-Riemannian structure.
This is indeed the case if we consider
the standard sub-Riemannian structures on
$\SL$ where the flow of the Reeb vector field preserves the
distribution and the fibre inner product.

The special linear group $\SL$ of degree two over the field $\R$ is
the Lie group of $2\times 2$ matrices with determinant $1$, that is,
\begin{equation*}
  \SL=\left\{
    \begin{pmatrix}
      x & y\\
      z & w
    \end{pmatrix}
    \colon
    x,y,z,w\in\R\mbox{ with }xw-yz=1
  \right\}\;,
\end{equation*}
where the group operation is taken to be matrix multiplication.
The Lie algebra $\mathfrak{sl}(2,\R)$ of $\SL$ is the algebra of
traceless  $2\times 2$ real matrices. A basis of
$\mathfrak{sl}(2,\R)$ is formed by the three matrices
\begin{equation*}
  p=\frac{1}{2}
  \begin{pmatrix}
    1 & 0 \\ 0 & -1
  \end{pmatrix}\;,\quad
  q=\frac{1}{2}
  \begin{pmatrix}
    0 & 1 \\ 1 & 0
  \end{pmatrix}
  \quad\mbox{and}\quad
  j=\frac{1}{2}
  \begin{pmatrix}
    0 & 1 \\ -1 & 0
  \end{pmatrix}\;,
\end{equation*}
whose corresponding left-invariant vector fields on $\SL$ are
\begin{align*}
  X&=\frac{1}{2}\left(
     x\frac{\pt}{\pt x} - y\frac{\pt}{\pt y}
    +z\frac{\pt}{\pt z} - w\frac{\pt}{\pt w}
  \right)\;,\\
  Y&=\frac{1}{2}\left(
     y\frac{\pt}{\pt x} + x\frac{\pt}{\pt y}
    +w\frac{\pt}{\pt z} + z\frac{\pt}{\pt w}
  \right)\;,\\
  K&=\frac{1}{2}\left(
    -y\frac{\pt}{\pt x} + x\frac{\pt}{\pt y}
    -w\frac{\pt}{\pt z} + z\frac{\pt}{\pt w}
  \right)\;.
\end{align*}
These vector fields satisfy the commutation relations
$[X,Y]=K$, $[X,K]=Y$ and $[Y,K]=-X$. For
$k\in\R$ with $k>0$, we
equip $\SL$ with the sub-Riemannian structure
obtain by considering the distribution $D$
spanned by $X_1=2kX$ and $X_2=2kY$ as well as the fibre inner product 
uniquely given by requiring $(X_1,X_2)$ to be a global
orthonormal frame. The appropriately normalised contact form
corresponding to this choice is
\begin{equation*}
  \omega=\frac{1}{4k^2}
  \left(z\dd x + w\dd y - x\dd z - y\dd w\right)\;,
\end{equation*}
and the Reeb vector field $X_0$ associated with the contact form
$\omega$ satisfies
\begin{equation*}
  X_0=[X_1,X_2]=4k^2K
  =2k^2\left(
    -y\frac{\pt}{\pt x} + x\frac{\pt}{\pt y}
    -w\frac{\pt}{\pt z} + z\frac{\pt}{\pt w}
  \right)\;.
\end{equation*}
The plane in $\SL$ passing
tangentially to the contact distribution through the identity
element
is the surface $S$ given as~(\ref{defn:Swu}) by the function
$u\colon\SL\to\R$ defined by
\begin{equation*}
  u(x,y,z,w)=y-z\;.
\end{equation*}
Observe that, on $S$, we have the relation $xw=1+y^2\geq 1$. Therefore, if a
point $(x,y,z,w)$ lies on the surface $S$ then so does the point
$(-x,y,z,-w)$, and neither $x$ nor $w$ can vanish on $S$. Thus,
the function $u\colon\SL\to\R$ induces a surface consisting of two
sheets. By symmetry, we restrict our attention to the sheet
containing the $2\times 2$ identity matrix, henceforth referred to as the
upper sheet. We compute
\begin{equation*}
  (X_1u)(x,y,z,w)=-k\left(y+z\right)
  \quad\mbox{and}\quad
  (X_2u)(x,y,z,w)=k\left(x-w\right)\;,
\end{equation*}
as well as
\begin{equation*}
  (X_0u)(x,y,z,w)=2k^2\left(x+w\right)\;.
\end{equation*}
We note that
\begin{equation*}
  \left((X_1u)(x,y,z,w)\right)^2+
  \left((X_2u)(x,y,z,w)\right)^2
  =k^2\left(y+z\right)^2+k^2\left(x-w\right)^2
\end{equation*}
vanishes on $S$ if and only if $y=z=0$ and $x=w$. From
$xw=1+y^2$, it follows that the surface $S$ admits the two
characteristic points $(1,0,0,1)$ and $(-1,0,0,-1)$, that is, one
unique characteristic point on each sheet.
Following Rogers and Williams~\cite[Section~V.36]{volume2}, we choose
coordinates
$(r,\theta)$ with $r>0$ and $\theta\in[0,2\pi)$ on the upper sheet of
$S\setminus\car(S)$ such that
\begin{align*}
  x&=\cosh\left(kr\right)+\sinh\left(kr\right)\cos(\theta)\;,\\
  w&=\cosh\left(kr\right)-\sinh\left(kr\right)\cos(\theta)\;,
     \quad\mbox{and}\\
  y&=\sinh\left(kr\right)\sin(\theta)\;.
\end{align*}
On the upper sheet of $S\setminus\car(S)$, we obtain
\begin{equation*}
  (X_1u)(r,\theta)=-2k\sinh\left(kr\right)\sin(\theta)
  \quad\mbox{and}\quad
  (X_2u)(r,\theta)=2k\sinh\left(kr\right)\cos(\theta)\;,
\end{equation*}
which yields
\begin{equation*}
  \sqrt{\left((X_1u)(r,\theta)\right)^2+\left((X_2u)(r,\theta)\right)^2}
  =2k\sinh\left(kr\right)\;,
\end{equation*}
as well as
\begin{equation*}
  (X_0u)(r,\theta)=4k^2\cosh\left(kr\right)\;.
\end{equation*}
A direct computation shows that on the upper sheet of
$S\setminus\car(S)$, we have
\begin{equation*}
  \widehat{X}_S=\frac{\pt}{\pt r}
  \quad\mbox{and}\quad
  \drift(r,\theta)=2k\coth\left(kr\right)\;,
\end{equation*}
which implies that
\begin{equation*}
  \frac{1}{2}\Delta_0=
  \frac{1}{2}\frac{\pt^2}{\pt r^2}+
  k\coth\left(kr\right)\frac{\pt}{\pt r}\;.
\end{equation*}
Hence, we recover all hyperbolic Bessel processes of order $3$ as
the canonical stochastic processes moving along the leaves of the
characteristic foliation of the upper sheet of
$S\setminus\car(S)$, and similarly on its lower sheet.
For further discussions on hyperbolic Bessel processes, see
Borodin~\cite{borodin}, Gruet~\cite{gruet00}, Jakubowski
and Wi\'{s}niewolski~\cite{jakubowski}, and Revuz and
Yor~\cite[Exercise~3.19]{revuz}.
As for
the Bessel process of order $3$ and the Legendre processes of order $3$, the
hyperbolic Bessel processes of order $3$ can be defined as the radial
component of Brownian motion on three-dimensional hyperbolic spaces.

\subsection{A unified viewpoint}
\label{s:unif}
The surfaces considered in the last two examples together with the
plane $\{z=0\}$ in the Heisenberg group are particular cases of the
following construction.

Let $G$ be a three-dimensional Lie group endowed with a contact
sub-Riemannian structure whose distribution $D$ is spanned by two
left-invariant vector fields $X_{1}$ and $X_{2}$ which are orthonormal
for the fibre inner product $g$ defined on $D$. Assume that the
commutation relations between
$X_{1},X_{2}$ and the Reeb vector field $X_0$ are given by, for some
$\kappa\in\R$,
\begin{equation*}
  [X_{1},X_{2}]=X_{0}\;,\quad
  [X_{0},X_{1}]=\kappa X_{2}\;, \quad
  [X_{0},X_{2}]=-\kappa X_{1}\;.
\end{equation*}
Under these assumptions the flow of the Reeb vector field $X_{0}$
preserves not only the 
distribution, namely $\e^{tX_{0}}_{*}D=D$, but also the fibre inner
product $g$. The examples presented in Section~\ref{sec:para} and in
Sections~\ref{sec:SU} and \ref{sec:SL}
satisfy the above commutation relations with
$\kappa=0$ in the Heisenberg group, and for a parameter
$k>0$, with $\kappa=4k^{2}$ in $\SU$ and $\kappa=-4k^{2}$
in $\SL$. These are the three classes of model spaces for three-dimensional
sub-Riemannian structures on Lie groups with respect to local
sub-Riemannian isometries, see for instance \cite[Chapter 17]{ABB} and
\cite{AB3D} for more details. 

In each of the examples concerned, the surface $S$ that we consider can be
parameterised as
\begin{align*}
  S&=\left\{\exp(x_{1}X_{1}+x_{2}X_{2}) :  x_{1},x_{2}\in
     \mathbb{R}\right\}\\
   &=\left\{\exp(r\cos \theta X_{1}+r\sin \theta X_{2}) :  r\geq 0,
     \theta\in [0,2\pi)\right\}\;.
\end{align*}
Observe that $S$ is automatically smooth, connected, and contains the
origin of the group.
Under these assumptions, the sub-Riemannian structure is of type
$\mathbf{d}\oplus\mathbf{s}$ in the sense
of~\cite[Section 7.7.1]{ABB}, and for $\theta$ fixed, the
curve $r\mapsto \exp(r\cos \theta X_{1}+r\sin \theta X_{2})$ is a
geodesic parameterised by length. Hence, $r\geq 0$ is the arc length parameter
along the corresponding trajectory. It follows that the surface $S$ is
ruled by geodesics, each of them having vertical component of the
initial covector equal to zero. We refer to \cite[Chapter 7]{ABB} for
more details on explicit expressions for sub-Riemannian geodesics in
these cases, see also \cite{BR3D}.

\bibliographystyle{plain}
\bibliography{references}

\end{document}